\documentclass[10pt]{article}
\usepackage{amsmath,amssymb,amsthm, color,geometry, verbatim, authblk}
\numberwithin{equation}{section}
\newtheorem{theorem}{Theorem}[section]
\newtheorem{corollary}{Corollary}[section]
\newtheorem{lemma}[theorem]{Lemma}
\newtheorem{proposition}{Proposition}[section]
\newtheorem{remark}{Remark}

\let \ga=\gamma

\let \e=\varepsilon

\makeatletter

\newcommand{\Rmnum}[1]{\expandafter\@slowromancap\romannumeral #1@}
\makeatother
\newcommand \nc{\newcommand}
\nc{\ba}{\begin{array}}\nc{\ea}{\end{array}}
\nc{\be}{\begin{eqnarray}}\nc{\ee}{\end{eqnarray}}
\nc{\beq}{\begin{equation}}\nc{\eeq}{\end{equation}}
\nc{\bex}{\begin{eqnarray*}}\nc{\eex}{\end{eqnarray*}}

\begin{document}
	
\makeatletter
\renewcommand{\@maketitle}{%
	\newpage
	\null
	\vskip 2em%
	\begin{center}%
		\let \footnote \thanks
		{\LARGE \@title \par}%
		\vskip 1.5em%
		{\large
			\lineskip .5em%
			\@author\par%
		}%
		\vskip 1em%
		{\large \@date}%
	\end{center}%
	\par
	\vskip 1.5em
}
\makeatother

\title{Global spherically symmetric solutions to the isothermal compressible Navier-Stokes equations with far-field vacuum}
\author{
	Xingyang Zhang\thanks{Department of Applied Mathematics, The Hong Kong Polytechnic University, Hong Kong, China; Email: mathzhangxy@outlook.com.}}
\date{}

	\maketitle

\begin{abstract}
In this paper, we consider the global spherically symmetric strong solutions to the compressible Navier-Stokes equations with far-field vacuum and density-dependent degenerate viscosity, following the framework proposed by Bresch-Vasseur-Yu \cite{B-V-Y 2021}. For the 1D Navier-Stokes equations, Wen-Zhang  \cite{W-Z SIAM 2025} considered the Cauchy problem which established the dependence relationship $\gamma-\delta-\frac{1}{p}\ge0$
within the $W^{2,p}(\mathbb{R})$ and $p\ge 2$. 
In this paper, we establish the global existence and uniqueness of strong solutions in $H^2([a,+\infty))$, $a>0$. In particular, we remove the restriction relating ($\gamma$, $\delta$, $p$), and instead assume that $\delta > 0.7427$. This result can be regarded as the first one on spherically symmetric strong solutions to the 3D Navier-Stokes equations with density-dependent viscosity proposed in \cite{B-V-Y 2021} and far-field vacuum.
\end{abstract}

{\noindent \textbf{Keywords:} Compressible Navier-Stokes equations; far-field vacuum; global existence; large data; spherically symmetric solutions.}
	
	\vspace{4mm}
	{\noindent\textbf{AMS Subject Classification (2020):} 76N05, 76N10, 35Q30.}

\tableofcontents

\section{Introduction}
In this paper, we study the compressible Navier-Stokes equations in $\Omega=\{\vec{x}\in \mathbb{R}^3,\ |\vec{x}|\ge a\}$ where $a>0$ is a constant, namely,
\begin{equation}\label{system}
	\left\{
	\begin{aligned}
		&\rho_t + {\rm{div}} (\rho  \vec{u}) = 0,\\[2mm]
		&(\rho \vec{u})_t+ {\rm{div}} (\rho \vec{u} \otimes \vec{u})+\nabla P(\rho) =2{\rm div}(\mu(\rho)\mathbb{D}\vec{u})+\nabla(\lambda(\rho) {\rm div}\vec{u}),
	\end{aligned}
	\right.
\end{equation} where $\rho (\vec{x},t)\ge 0$ and $\vec{u}(\vec{x},t)$ define for $(\vec{x},t) \in \Omega \times [0,T]$,  and $P(\rho)=A_1\rho^\gamma$ for constants $\gamma \ge 1$ and $A_1 > 0$ denote the density, velocity and pressure function, $\mathbb{D}\vec{u}=\frac{1}{2}(\nabla \vec{u}+\nabla^t \vec{u})$ denotes the strain tensor, $\mu(\rho)=A_2\rho^{\delta}$ and $\lambda(\rho)=2A_2(\delta-1)\rho^\delta$ denote shear and bulk viscosities (we assume $A_1=A_2=1$ without loss of generality) satisfying the BD relation $$\lambda(\rho)=2(\mu'(\rho)\rho-\mu(\rho))$$
and
$$
\mu(\rho)>0, \ \ 2\mu(\rho)+3\lambda(\rho)\ge 0
$$
where $\delta \in [\frac{2}{3},1)$.

Many results have been obtained for the compressible Navier-Stokes equations. First, we review some early results. Kazhikhov and Shelukhin \cite{K-S} proved the global existence and uniqueness of strong solutions to the 1D problem with arbitrarily large initial data, under the assumptions of no vacuum and constant viscosity coefficients. Subsequently, Vaigant and Kazhikhov \cite{Vaigant} obtained the global existence and uniqueness of solutions to the 2D problem when the shear viscosity is constant and the bulk viscosity takes the form $\rho^\beta$ ($\beta>3$). 
Later, the results for $\beta>\frac{4}{3}$ in the whole space and bounded domains were presented in \cite{Huang-Li, Jiu-Wang-Xin} and \cite{Fan-Li-Li}, respectively.

For the 3D problem, Lions \cite{Lions} first established the global existence of weak solutions to the constant viscosity problem with large initial data when the adiabatic exponent $\gamma\ge\frac{9}{5}$. 
Subsequently, Feireisl, Novotn\'y and Petzeltov\'a \cite{Feireisl 2001} and Jiang and Zhang \cite{Jiang-Zhang} derived the existence of weak solutions for $\gamma>\frac{3}{2}$ and spherically symmetric weak solutions for $\gamma>1$, respectively. 
The uniqueness of the solutions obtained by Feireisl- Novotn\'y-Petzeltov\'a and by Jiang-Zhang remains a challenging open problem. 
Hu \cite{Hu2,Hu1} investigated the existence of Lions-Feireisl solutions for $\gamma\in[1,\frac{n}{2}]$ ($n=2,3$). 
Additionally, some known results exist for the global existence and uniqueness of solutions to three-dimensional problems under certain smallness assumptions, and details can be found in \cite{Hoff 1995, Hong-Hou-Peng-Zhu, Huang-Li-Xin, M-N, W-Z SIAM 2017}.

When the viscosity depends on density, the degeneracy of the equations becomes the main difficulty. 
In fact, the density-dependent viscosity can be explained by the derivation of the isentropic Navier-Stokes equations from the Boltzmann equation, please refer to \cite{Chapman-C, L-X-Y 1998} for details. 
Yang and Zhu \cite{Yang-Zhu} obtained the global existence of solutions to the 1D problem with initial data of compact or non-compact support. 
For initial density strictly positive (non-compact support and positive at $x=\pm\infty$), Mellet and Vasseur \cite{M-V 2008} proved the global existence of strong solutions to the Cauchy problem with large initial data. 
A key step in \cite{M-V 2008} is to obtain a positive lower bound for the density via the BD entropy estimate (see \cite{BD 2004, BD entropy} for details), which requires the density-dependent viscosity to satisfy:
\begin{equation}\label{Q3-mu}
	\mu(\rho)= \widetilde{\varepsilon}\rho^\delta,
\end{equation}
where $\widetilde{\varepsilon}>0$ is a constant and $\delta\in[0,\frac{1}{2})$. 
Jiu, Li and Ye \cite{Jiu JDE} considered $\mu(\rho)=1+\rho^\beta$ ($\beta \ge0 $) and established the global classical solutions. For the case $\delta=1$, Guo, Jiu and Xin \cite{GJX} considered the global existence of spherically symmetric weak solutions to the 3D problem. 
Subsequently, Haspot \cite{Haspot} extended the result of Mellet and Vasseur \cite{M-V 2008} from $\delta\in[0,\frac{1}{2})$ to $\delta\in(\frac{1}{2},1]$ by estimating the effective velocity and applying the maximum principle to the mass equation to derive a lower bound for the density. 
In Haspot's result \cite{Haspot}, due to the interaction between pressure and viscosity, the adiabatic exponent $\gamma$ must satisfy:
\begin{equation}\label{Q3-ga1}
	\gamma\ge\delta+\frac{1}{2}+\epsilon,\quad \text{with}\quad \epsilon\in(0,\frac{1}{4}).
\end{equation}

For initial data without compact support and allowing vacuum at $x=\pm\infty$, Cao, Li and Zhu \cite{C-L-Z} proved the global existence and uniqueness of strong solutions to the Cauchy problem, where the adiabatic exponent and the density-dependent viscosity satisfy:
\begin{equation}\label{Q3-ga2}
	\begin{cases}
		\gamma>1\ \text{and}\ \gamma\ge\delta+\frac{1}{2}, & \text{if}\ \delta\in(0,1);\\[2mm]
		\gamma>\frac{3}{2}, & \text{if}\ \delta=1,
	\end{cases}
\end{equation}
The necessity of the technical condition (\ref{Q3-ga2}) on the adiabatic exponent is similar to that in \cite{Haspot}, for the case of symmetric solutions when $\delta=1$, please refer to \cite{C-L-Z CVPDE}. Recently, Wen and Zhang \cite{W-Z DCDS} considered the two-fluid system with $\delta=1$ and $\gamma^\pm>1$ and proved that the global existence and uniqueness of the strong solutions to the Cauchy problem in $W^{2,p}(\mathbb{R})$ ($p\ge 2$), where $\gamma^\pm$ and $p$ satisfy
\begin{equation}\label{wz cond1}
\displaystyle \min \{\gamma^+,\gamma^-\}-1-\frac{1}{p}\ge 0.
\end{equation} The results in \cite{W-Z DCDS} can be extended to the Navier-Stokes equations, and for the case $0 < \delta < 1$ refer to \cite{W-Z SIAM 2025}. Notably, Wen-Zhang \cite{W-Z SIAM 2025} allowed $\gamma = 1$, where $\gamma$, $p$ and $\delta$ also satisfy
\begin{equation}\label{wz cond2}
\displaystyle \gamma-\delta-\frac{1}{p}\ge 0.
\end{equation}

For the general compressible Navier-Stokes equations with density-dependent viscosity, the result for $\gamma>1$ remains an open problem. 
Regarding high-dimensional problems, Vasseur and Yu \cite{V-Y 2016} and Li and Xin \cite{Li-Xin} solved the global existence of weak solutions for the 3D and 2D cases, respectively. Further research can be found in the work of Bresch, Vasseur and Yu \cite{B-V-Y 2021}. For the results of strong solutions, 
Xin and Zhu \cite{X-Z adv, X-Z JMPA} established the global 
well-posedness for $\delta>1$ and the local well-posedness for $0<\delta<1$ of the three-dimensional isentropic compressible 
Navier-Stokes equations with degenerate viscosities under suitable 
structural assumptions on $(\gamma,\delta)$ and the initial data; 
the corresponding result for the full (non-isentropic) system was obtained by Duan, Xin and Zhu \cite{Q-X-Z ARMA}.

In this paper, we consider the global spherically symmetric strong solutions to the compressible Navier-Stokes equations with far-field vacuum and density-dependent degenerate viscosity. Furthermore, we remove the technical condition between $\gamma$ and $\delta$. This paper is organized as follows: In Section 1, we establish the spherically symmetric system and introduce the main results. In Section 2, we prove the local existence and uniqueness of solutions. By removing the vacuum and constructing an auxiliary system, we establish uniform estimates and prove the strong convergence of the corresponding iterative scheme, which together with the initial data leads to the local existence and uniqueness of solutions to the original system in the presence of vacuum. In Section 3, based on the local existence result from Section 2 and the conditions of $\delta$, we derive the global {\it{a priori}} estimates and thereby complete the proof of our main theorem.


\bigskip
	
Before stating our main result, we would like to introduce some notations and definitions which will be used throughout this paper.

\subsection*{Notations}
\begin{itemize}
\item $I_a \triangleq[a,+\infty);$

\item $L^l=L^l(I_a),  \ D^{k,l}=\left\{ u\in L^1_{\rm{loc}}(I_a): \|\partial^k_x u \|_{L^l}<\infty\right\},$\ $W^{k,l}=L^l\cap D^{k,l};$

\item  $r=|\vec{x}|$, \ $\displaystyle\int = \int_{I_a}\,dr.$
		
\end{itemize}

\subsection{Symmetric system}
We assume $\vec{u}(\vec{x},t)=u(r,t)\frac{\vec{x}}{r}$, $\rho(\vec{x},t)=\rho(r,t)$ and $r\in I_a$, and we have
\begin{equation*}
	\begin{cases}
	\displaystyle [2{\rm div}(\mu(\rho)\mathbb{D}\vec{u})]^r=\frac{2}{r^2}(r^2\rho^{\delta}u_r)_r-\frac{4\rho^\delta u}{r^2},\\[2mm]
	\displaystyle [\nabla(\lambda(\rho) {\rm div}\vec{u})]^r=2(\delta-1)[\rho^\delta(u_r+\frac{2}{r}u)]_r.
	\end{cases}
\end{equation*}

Then we obtain the spherically symmetric form of system \eqref{system} as follows:
\begin{equation}\label{system1D}
	\left\{
	\begin{aligned}
		&\rho_t + \frac{1}{r^2} \big( r^2 \rho u \big)_r = 0, \\
		&(\rho u)_t + \frac{1}{r^2} \big( r^2 \rho u^2 \big)_r + P(\rho)_r= 2\delta \rho^\delta (u_r+\frac{2}{r} u)_r +2\delta \rho^{\delta-1}\rho_r u_r+2\delta(\delta-1)\rho^{\delta-1}\rho_r (u_r+\frac{2}{r}u).
	\end{aligned}
	\right.
\end{equation}

\eqref{system1D} is complemented by the following conditions
\begin{align}\label{initial}
	\begin{cases}
			(\rho,u)(r, 0) = (\rho_0,u_0)(r), \  r\in I_a,\\
			(\rho,u)\to (0,0), \ {\rm{as}}\  r \to +\infty,\ t \ge 0,
	\end{cases}
\end{align}
and
\begin{align}\label{boundary}
u(r,t)|_{r=a}=0, \ t\ge 0.
\end{align}

\subsection{Main result}

Now we are ready to state our main result.
\begin{theorem}\label{main result}
For any given $\gamma\in[1,\infty)$ and $\delta$ satisfying the following conditions
\begin{equation}\label{condi of delta}
	\begin{cases}
		\displaystyle	\frac{2}{3}\leq \delta < 1 ,\\[2mm]
				\displaystyle \frac{2\delta(2\delta-1)}{(1-\delta)^2} \triangleq K(\delta),\ \ 	\frac{4-2\delta}{1-\delta}\leq \frac{K(\delta)+\sqrt{K(\delta)^2-4K(\delta)}}{2}.
	\end{cases} 
\end{equation} 	and,we assume that the initial data ($\rho_0$, $u_0$) satisfy
\begin{equation}\label{ini data}
0<\rho_0\in L^\infty,\ r^{2+\alpha(\gamma)}\rho_0\in L^1, \ r u_0\in H^2\ \ \text{and}\ \ r(\rho_0^{\delta-1})_r\in H^1,
\end{equation}
where
$$\alpha(\gamma)=\left\{\begin{array}{l}
\displaystyle\alpha(1)\in(1,2), \ \ \ {\rm if} \ \ \ \gamma=1, \\
[3mm] \displaystyle 0, \ \ \ {\rm if} \ \ \ \gamma>1,
\end{array}
\right.
$$
and that the compatibility condition
\begin{equation}\label{compatibility condition}
2(u_{0,r}+\frac{2}{r}u_0)_r=\rho_0^{1-\delta}g,\ \ 
\end{equation}
holds for some $g\in L^2$. Then for any $T>0$ there exists a unique strong solution $(\rho,u)$ to $\eqref{system1D}$-$(\ref{initial})$ over $I_a\times [0,T]$, satisfying
\begin{equation}\label{regularity}
\begin{split}
&r^{2+\alpha(\gamma)}\rho \in C([0,T];L^1),\ \ r\rho \in C([0,T];H^2),\ \  r\rho_t \in C([0,T];H^1),  \\[2mm]
&r(\rho^{\delta-1})_r\in C([0,T];H^1),\ \ r(\rho^{\delta-1})_{rt}\in C([0,T];L^2),\\[2mm]
&ru \in C([0,T];H^2)\cap L^2([0,T];D^3) ,\ \ ru_t \in C([0,T];L^2) \cap L^2([0,T];D^1).\\
\end{split}
\end{equation}
\end{theorem}

\begin{remark}
	By numerical computation (e.g., bisection method), condition \eqref{condi of delta} implies $1>\delta > 0.7427$. 
	We can verify that when $\delta = 0.7427$, $\displaystyle \frac{4-2\delta}{1-\delta} \approx 9.7730$ and $\displaystyle \frac{K+\sqrt{K^2-4K}}{2} \approx 9.7770 $, which satisfy condition \eqref{condi of delta}; when $\delta \to 1$, 
	\begin{equation*}
		\begin{cases}
			\displaystyle \frac{4-2\delta}{1-\delta} \sim\frac{1}{1-\delta},\\[2mm]
			\displaystyle \frac{K+\sqrt{K^2-4K}}{2} \sim\frac{1}{(1-\delta)^2}, 
		\end{cases}
	\end{equation*} which also satisfy condition \eqref{condi of delta}.
\end{remark}


\subsection{Main challenges and ideas}


Unlike the case when $\delta=1$ \cite{C-L-Z CVPDE}, when $0<\delta<1$ and $\rho \to 0$ in the far-field, the viscosity of the momentum equation exhibits singularity
\begin{equation*}
	\begin{cases}
		u_t \sim u_{rr}, \ \ \delta=1,\\
		u_t \sim \rho^{\delta-1}u_{rr}, \ \ \delta<1,
	\end{cases}
\end{equation*} which constitutes the essential difference and difficulty.

Inspired by \cite{BD 2004, BD entropy}, we introduce the effective velocity 
$$
v = u + 2\delta \rho^{\delta-2}\rho_r
$$
to reformulate the momentum equation. In \cite{W-Z SIAM 2025}, we considered the 1D Cauchy problem and established global solutions in $W^{2,p}$ under the condition $\displaystyle \gamma - \delta - \frac{1}{p} > 0$. We rewrote the pressure term in the momentum equation as $\frac{1}{\rho}(\rho^\gamma)_x \sim \rho^{\gamma-\delta}(v-u)$. By choosing a sufficiently large $p$, we derived the estimate for $\Vert v\Vert_{ L^\infty}$ with $\rho$-weighted estimates of $(u, v)$ and completed the estimate of $\Vert u \Vert_{L^p}$ via $$
\int |\rho^{\gamma-\delta}(v-u) |u|^{p-2} u|\leq C\Vert \rho^{\frac{1}{p}}\Vert_{ L^p}\Vert \rho^{\gamma-\delta-\frac{1}{p}} \Vert_{ L^\infty}\Vert v \Vert_{ L^\infty} \Vert u \Vert_{ L^p}^{p-1}+... .
$$
In this paper, we still require similar $\rho$-weighted estimates of $(u, v)$ ( $\Vert r^{\frac{2}{p}}\rho^{\frac{1}{p}} u \Vert_{ L^p}$ and $\Vert r^{\frac{2}{p}}\rho^{\frac{1}{p}} v \Vert_{ L^p}$ ) to obtain the estimate of $\Vert ru \Vert_{ L^2}$. However, due to the difficulties caused by $\delta < 1$ and the higher-dimensional system, $p$ has an upper bound depending on $\delta$ as follow:
\begin{equation*}
	\displaystyle \frac{2\delta(2\delta-1)}{(1-\delta)^2} \triangleq K(\delta),\ \ 	2\leq p \leq \frac{K(\delta)+\sqrt{K(\delta)^2-4K(\delta)}}{2}.
\end{equation*} As a consequence, $p$ cannot be chosen large enough to ensure a solution exists for all $\gamma \ge 1$. This is the main difficulty and difference between the 3D spherically symmetric problem and the 1D case.

To overcome this, by some refined energy estimates we choose a suitable $p$ satisfying $$p=\frac{4-2\delta}{1-\delta}$$ to establish the estimate of $\Vert ru \Vert_{ L^2}$ by $\rho$-weighted estimates of $(u,v)$, then we only require $$2\leq p=\frac{4-2\delta}{1-\delta}\leq \frac{K(\delta)+\sqrt{K(\delta)^2-4K(\delta)}}{2}, $$ please refer to Lemma \ref{le:rho up+vp}-\ref{le:uLP} for more details. The higher-order estimates of $u$ can be obtained directly on the original equation, without resorting to the effective velocity $v$ and its higher-order estimates. Consequently, we remove the restriction relating ($\gamma$, $\delta$, $p$). Then we can establish global solutions in $H^2$, and this result also allows the isothermal flow.


As in \cite{W-Z SIAM 2025}, since the viscosity coefficient $\rho^{\delta-1}$ becomes singular as $\rho \to 0$ in the far field, to construct a local-in-time solution we first regularize the problem by imposing a positive lower bound on the initial density (i.e.\ removing the vacuum). To obtain estimates uniform with the artificial lower bound, we rewrite the equations and introduce an new auxiliary system.
 Note that, for the local {\it{a priori}} estimates of $\rho^{\delta-1}$, we need to construct a linear known function of it to obtain some auxiliary estimates, and then, complete the proof of strong convergence of the iterative scheme in the no-vacuum regime. Finally, using the regularity of the initial data we transfer the solution of the auxiliary system back to the original one, and allow the far-field vacuum by letting the artificial lower bound $\eta \to 0$.

\section{Local-in-time existence and uniqueness}\label{sec2}
In this section, we establish the local-in-time existence and uniqueness of strong solutions to the initial-boundary value problem (\ref{system1D})-(\ref{boundary}).

\begin{theorem}\label{localthm}
For any $\gamma \ge1$, assume that the initial data $(\rho_0, u_0)$ satisfies $\rho_0 > 0$, $(r\rho_0, ru_0) \in H^2$, $r(\rho_0)^{\delta-1}  \in D^1\cap D^2$, and the initial compatibility condition (\ref{compatibility condition}). Then there exists a time $T^*>0$ such that (\ref{system1D})-(\ref{boundary}) admits a unique strong solution $(\rho, u)$ satisfying
\begin{equation}\label{regularity-2}
\begin{split}
&0<r\rho,\ \ r\rho \in C([0,T^*];H^2),\ \  r\rho_t \in C([0,T^*];H^1),  \\[2mm]
&r(\rho^{\delta-1})_r\in C([0,T^*];H^1),\ \ r(\rho^{\delta-1})_{rt}\in C([0,T^*];L^2),\\[2mm]
&ru \in C([0,T^*];H^2)\cap L^2([0,T^*];D^3),\ \ ru_t \in C([0,T^*];L^2)\cap L^2([0,T^*];D^1).\\
\end{split}
\end{equation}
\end{theorem}

\subsection{An auxiliary system with positive initial density}
Motivated by the formulation of BD entropy \cite{BD 2004, BD entropy} and the idea of constructing an auxiliary system in \cite{W-Z SIAM 2025}, we introduce the effective velocity $v=u+2\delta\rho^{\delta-2}\rho_r$ and two new variables: $\phi=\rho^{\gamma-\delta}$ and $h=2\rho^{\delta-1}$. Then, system \eqref{system1D} can be formally rewritten as follows:
\begin{equation}\label{reform sys}
\left\{
\begin{aligned}
&\rho_t+\rho_r u+\rho u_r+\frac{2}{r}\rho u=0,\\
&h_t+h_ru+(\delta-1)hu_r+\frac{2}{r}(\delta-1)hu=0,\\
&\phi_t+ \phi_r u+(\gamma-\delta)\phi u_r+\frac{2}{r}(\gamma-\delta)\phi u=0,\\
&v_t+uv_r+\frac{\gamma}{2\delta} \phi (v-u)=0,\\
&u_t+uu_r+\frac{\gamma}{2\delta}  \phi (v-u)=\delta h(u_r+\frac{2}{r}u)_r+\delta(v-u)u_r+(\delta-1)(v-u)\frac{2}{r}u.
\end{aligned}
\right.
\end{equation} which is complemented by some initial-boundary conditions:
\begin{equation}\label{lin initial}
\begin{split}
(\rho,h,\phi, v,u)(r,0)=(\tilde \rho_0,\tilde h_0,\tilde \phi_0,\tilde v_0=u_0+\frac{\delta}{\delta-1}\tilde h_{0,r}, u_0)(r),
\end{split}
\end{equation}
\begin{equation}\label{lin boundary}
	u|_{r=a}=0,
\end{equation} and
\begin{equation*}
	(\rho,\phi,h,u) \to (\eta,\eta^{\gamma-\delta},2\eta^{\delta-1}, 0  ) \ \ \text{as}\ \  r \to +\infty,
\end{equation*} where
\begin{equation}\label{ini 22}
\tilde \rho_0=\rho_0+ \eta,\ \ \tilde h_0=2\tilde \rho_0^{\delta-1},\ \ \tilde \phi_0=\tilde \rho_0^{\gamma-\delta} \ \ \text{(the constant $\eta$ satisfies} \ \ 0<\eta<1 ),
\end{equation} and the initial compatibility conditions
\begin{equation}\label{compatibility condition local}
	\partial_r(u_{0,r}+\frac{2}{r}u_0)=(\tilde h_0)^{-1}\tilde g,\ \  
\end{equation}
for some $\tilde g\in L^2$, which satisfies the following condition:
$$
\tilde g=\frac{\tilde h_0}{h_0}g.
$$

The assumptions of Theorem \ref{localthm}, together with \eqref{lin initial} and \eqref{ini 22}, imply that
\begin{equation}\label{new ini}
\begin{cases}
r(\tilde\rho_0-\eta) \in H^2,\  \tilde \phi_0 > \eta^{\gamma-\delta},  \ 2\eta^{\delta-1}>\tilde h_0> \Vert \tilde \rho_0 \Vert_{ L^\infty}^{\delta-1},\  ru_0 \in H^2,\  r\tilde \phi_{0,r}\in H^1, \ \tilde \phi_0\in L^\infty,\\[2mm]
 r\tilde h_{0,r}  \in H^1,\ 
 \tilde \phi_0\tilde h_0=2(\tilde \rho_0)^{\gamma-1} \in L^\infty, \ r\tilde \phi_0\tilde h_0 \in D^1,\ \  r\tilde v_0\in H^1.\
\end{cases}
\end{equation}

Now, we establish the local-in-time existence and uniqueness of strong solutions to (\ref{reform sys})-(\ref{compatibility condition local}).

\begin{theorem}\label{localthm2}
Under the conditions of (\ref{new ini}), there exists a time $T^*>0$ such that the initial-boundary value problem (\ref{reform sys})-(\ref{compatibility condition local}) admits a unique strong solution $(\rho,h,\phi, v,u)$ satisfying
\begin{equation}\label{regularity-3}
\begin{split}
&r (\rho-\eta) \in C([0,T^*];H^2),\ \  r\rho_t \in C([0,T^*];H^1),  \ \  \phi \in C([0,T^*];L^\infty)\\
&(r\phi_r, rh_r)\in C([0,T^*];H^1),\ \ r\phi_t \in C([0,T^*];H^1), \ \ rh_t \in C([0,T^*];D^1),  \\
&r v \in C([0,T^*];H^1),\ \  rv_t \in C([0,T^*];L^2),  \\
&ru \in C([0,T^*];H^2)\cap L^2([0,T^*];D^3),\ \ ru_t \in C([0,T^*];L^2)\cap L^2([0,T^*];D^1).\\
\end{split}
\end{equation} 
\end{theorem}

\begin{remark}
According to Theorem \ref{localthm2}, we can derive the local-in-time existence and uniqueness results stated in Theorem \ref{localthm}. For the detailed proof, see Sections \ref{sec 2.5} and \ref{sec 2.6}.
\end{remark}

\subsection{A linearized system with $\eta>0$}\label{sec2.1}
In this section, we consider the following linearized system,
\begin{equation}\label{lin sys}
	\left\{
	\begin{aligned}
		&\rho_t+\rho_r U+\rho U_r+\frac{2}{r}\rho U=0,\\
		&h_t+h_rU+(\delta-1)hU_r+\frac{2}{r}(\delta-1)hU=0,\\
		&\phi_t+\phi_rU+(\gamma-\delta)\phi U_r+\frac{2}{r}(\gamma-\delta)\phi U=0,\\
		&v_t+Uv_r+\frac{\gamma}{2\delta} \phi (v-U)=0,\\
		&u_t+UU_r+\frac{\gamma}{2\delta}\phi(v-U)=\delta h\left(u_r+\frac{2}{r}u\right)_r+\delta(v-U)u_r+(\delta-1)(v-U)\frac{2}{r}U,
	\end{aligned}
	\right.
\end{equation}
which is complemented by conditions (\ref{lin initial})-\eqref{lin boundary}, where $U(r,t)$ and $H(r,t)>0$ ($H(r,t)$ will be used in the proof of Lemma \ref{le:local rho W1p} later)  are given functions satisfying $U(r,0)=u_0$, $H(r,0)=\tilde h_0$ and $U(a,t)=0$. Note that the second equation of \eqref{lin sys}, together with $\tilde h_0 > 0$, implies $$h > 0$$. In addition, $(H,U)$ satisfy the following regularities: 
\begin{equation}\label{U}
\begin{split}
&rU \in C([0,T];H^2)\cap L^2([0,T];D^3),\ \ rU_t\in C([0,T];L^2)\cap L^2([0,T];D^1),\\
&r H_r \in L^\infty([0,T];H^1),\ \  r H_{rt} \in L^\infty([0,T];L^2),\ \  r \frac{H_t}{H}\in L^\infty([0,T];L^2), \ \  \frac{H_t}{H}\in L^\infty([0,T];L^\infty),\\
& rH\left(U_r+\frac{2}{r}U\right)_r \in L^\infty([0,T];L^2),
\end{split}
\end{equation}  for some $T>0$ which is to be determined later.

Under the condition $\tilde \rho_0 > \eta$, the linearized system \eqref{lin sys}–\eqref{U} is locally solvable, more precisely:
\begin{proposition}\label{prop1}
If the initial data satisfy
the conditions of Theorem \ref{localthm2} and the compatibility condition (\ref{compatibility condition local}), then (\ref{lin sys})-(\ref{U}) admits a unique strong solution satisfying 
\begin{equation*}
\begin{split}
&r(\rho-\eta) \in C([0, T];H^2),\ \  r\rho_t \in C([0, T];H^1),  \\
&   \phi \in C([0, T];L^\infty), \ \ r \phi_r \in C([0, T];H^1), \ \  r\phi_t \in C([0, T];H^1), \\
&h \in  C([0, T];L^\infty ),\ \ rh_r \in  C([0, T];H^1),\ \ rh_t\in C([0, T];H^1),  \\
&rv \in C([0,T];H^1),\ \ rv_t \in C([0, T];L^2),\\
&ru \in C([0,T];H^2) \cap L^2([0,T],D^3),\ \ ru_t \in C([0, T];L^2)\cap L^2([0,T],D^1),\\
\end{split}
\end{equation*}for any $T>0$. 
\end{proposition}

\begin{remark}
For more details of the proof, please refer to \cite{Cho04, Cho-Kim06}. 
\end{remark}

In the next subsection, we will establish the uniform {\it a priori} estimates independent of $\eta$, $\Vert h \Vert_{ L^\infty}$ and $\Vert r h_t \Vert_{ L^2}$ for the unique solution $(\rho, h,\phi, v, u)$.

\subsection{The uniform {\it a priori} estimates}\label{sec 2.3}
In this subsection, we will derive a local-in-time solution of (\ref{lin sys}) with \eqref{U} and \eqref{lin initial}-\eqref{compatibility condition local} by constructing an iteration scheme with uniform bounds, and then passing to the limit. Now, we first choose a positive constant $c_0$ large enough so that

\begin{equation}\label{co}
\begin{split}
&1+\eta+\Vert (\tilde\phi_0,(\tilde h_0)^{-1})\Vert_{ L^\infty}+\Vert \tilde\rho_0\Vert_{ L^\infty}+
\Vert ru_0\Vert_{ H^2}+
\Vert r\tilde{v}_0\Vert_{ H^1}+\Vert r(\tilde\rho_0-\eta)\Vert_{ H^2}+\Vert r(\tilde\phi_{0,r},\tilde h_{0,r}) \Vert_{ H^1}\\
&+\Vert \tilde h_0 \tilde \phi_0\Vert_{ L^\infty}+\Vert r(\tilde h_0 \tilde \phi_0)_r\Vert_{ L^2}+\Vert \tilde g \Vert_{ L^2}< c_0.
\end{split}
\end{equation}

We assume that there exist positive constants $c_i$ ($i=1,2,3$) depending on $c_0$ and some known parameters, such that

\begin{equation}\label{c1,c2}
\begin{cases}
\displaystyle \sup_{0\leq t \leq T}\Vert r U(\cdot,t) \Vert_{H^1} \leq c_1,\\
\displaystyle \sup_{0\leq t \leq T}(\Vert rU_t(\cdot,t)\Vert_{ L^2}+\Vert r U(\cdot,t) \Vert_{D^2} )+\int_0^T \left(\Vert r U_{rrr}(\cdot,t)\Vert_{ L^2}^2+\Vert r U_{rt}(\cdot,t)\Vert_{ L^2}^2\right) dt\leq c_2,\\
\displaystyle\sup_{0\leq t \leq T}\left(\Vert rH(\cdot,t)_r \Vert_{H^1}+\left\Vert r 
\frac{H_t}{H}(\cdot,t) \right\Vert_{L^2 }+\Vert r H(U_r+\frac{2}{r}U)_r(\cdot,t)\Vert_{L^2}\right)\leq c_2,\\
\displaystyle\sup_{0\leq t \leq T}\left(\Vert rH(\cdot,t)_{rt} \Vert_{L^2}+\left\Vert 
\frac{H_t}{H}(\cdot,t) \right\Vert_{L^\infty}\right)\leq c_3,\\
\end{cases}
\end{equation}	for $1<c_0\leq c_1 \leq c_2<c_3< \infty$, where $c_1$, $c_2$, $c_3$ and $T$ are determined by \eqref{def ci}.

In the following proposition, we will establish some local-in-time uniform estimates. Let $C \ge1$
denotes a generic positive constant depending on $c_0$, $\frac{1}{a}$  and some other known constants but independent of $c_1$, $c_2$, $c_3$, $t$ and solutions. $C_\eta$ depends on $C$ and $\eta$.

\begin{proposition}\label{prop2}
We assume $(U,H)$ satisfies (\ref{c1,c2}). Then there exists a time $T\le (Cc_3)^{-5}$ such that the solution obtained in Proposition \ref{prop1} satisfies that
\begin{equation}\label{prop2-1}
\begin{cases}
\displaystyle	\Vert r u(\cdot,t) \Vert_{H^1} \leq c_1,\\
\displaystyle	\Vert r u_t(\cdot,t) \Vert_{L^2}+\Vert r u(\cdot,t) \Vert_{D^2} +\int_0^T \left(\Vert r u_{rrr}(\cdot,t)\Vert_{ L^2}^2+\Vert r u_{rt}(\cdot,t)\Vert_{ L^2}^2\right) dt\leq c_2,\\
\displaystyle \Vert rh(\cdot,t)_r \Vert_{H^1}+\Vert r 
	\frac{h_t}{h}(\cdot,t) \Vert_{L^2 }+\Vert r h(u_r+\frac{2}{r}u)_r(\cdot,t)\Vert_{L^2}\leq c_2,\\
\displaystyle	\Vert rh(\cdot,t)_{rt} \Vert_{L^2}+\Vert 
	\frac{h_t}{h}(\cdot,t) \Vert_{L^\infty}\leq c_3,
	\end{cases}
\end{equation}and
\begin{equation*}
	\begin{cases}
		\Vert r (\rho-\eta)(\cdot,t)\Vert_{ H^2}\leq C,\\
		\Vert r\rho_t(\cdot,t)\Vert_{ H^1}\leq Cc_2,
	\end{cases}
\end{equation*} for all $t\in [0,T]$.

\end{proposition}

The proof of Proposition \ref{prop2} consists of Lemma \ref{le:local rho W1p}-\ref{le:local u}.

\begin{lemma}\label{le:local rho W1p}
Let ($\rho$, $h$, $\phi$, $v$, $u$) be the solution to (\ref{lin sys}) and (\ref{lin initial}) on $I_a \times [0,T]$ which is stated in Proposition \ref{prop1} and $(U,H)$ satisfies (\ref{c1,c2}). Then the following estimates hold
\begin{equation}\label{local phi W1p 1}
\begin{cases}
\displaystyle \rho(\cdot,t)>0,\ \ h(\cdot,t)>\frac{1}{C},\\
\displaystyle \Vert \phi(\cdot,t)\Vert_{ L^\infty}+\Vert r (\rho-\eta) (\cdot,t)\Vert_{ H^2}+\Vert r h_r (\cdot,t)\Vert_{ H^1}+\Vert r\phi_r(\cdot,t) \Vert_{ H^1}\\
\displaystyle +\left\Vert \frac{h}{H}(\cdot,t)\right\Vert_{ L^\infty} +\Vert h\phi(\cdot,t) \Vert_{ L^\infty}
+\Vert r(h\phi)_r(\cdot,t) \Vert_{ L^2}\leq C,\\
\displaystyle \Vert r \rho_t (\cdot,t)\Vert_{ L^2}+\Vert r \phi_t (\cdot,t)\Vert_{ L^2}+\left\Vert r \frac{h_t}{h}(\cdot,t)\right\Vert_{ L^2}\leq Cc_1,\\
\displaystyle \Vert r \rho_{rt}(\cdot,t)\Vert_{ L^2}+\Vert r h_{rt}(\cdot,t)\Vert_{ L^2}+\Vert r \phi_{rt}(\cdot,t)\Vert_{ L^2}+\left\Vert  \frac{h_t}{h}(\cdot,t)\right\Vert_{ L^\infty}\leq Cc_2,
\end{cases}
\end{equation}
for all 
\begin{equation}\label{t1}
0\leq t \leq T \leq (Cc_2)^{-2}.
\end{equation}
\end{lemma}
\begin{proof}
By virtue of $(\ref{lin sys})_1$, \eqref{co}, \eqref{c1,c2} and the standard method of characteristics, one can obtain that
\begin{equation}\label{lin rho H2}
	\Vert r (\rho-\eta) \Vert_{H^2}+\Vert \rho \Vert_{ L^\infty}\leq C, \ \ \rho(\cdot,t)>0,
\end{equation} and 
\begin{equation}\label{lin rhot}
	\begin{cases}
		 \Vert r\rho_t \Vert_{ L^2} \leq C(\Vert \rho_r \Vert_{ L^\infty}\Vert rU\Vert_{ L^2}+\Vert \rho \Vert_{ L^\infty}\Vert r U_r \Vert_{ L^2}+\Vert \rho \Vert_{ L^\infty}\Vert rU\Vert_{ L^2})\leq Cc_1,\\
		 \Vert r\rho_{rt} \Vert_{ L^2}\leq C(\Vert U \Vert_{ L^\infty}\Vert r \rho_{rr}\Vert_{ L^2}+\Vert \rho_r \Vert_{ L^\infty}\Vert rU_r\Vert_{ L^2}+\Vert \rho \Vert_{ L^\infty}\Vert r U_{rr}\Vert_{ L^2})\\
		 +C(\Vert \rho_r \Vert_{ L^\infty} \Vert r U\Vert_{ L^2}+\Vert \rho \Vert_{ L^\infty}\Vert rU_r \Vert_{ L^2})\leq Cc_2,
	\end{cases}
\end{equation} for all $0\leq t \leq T \leq (Cc_2)^{-2}$.

Similarly to \cite{W-Z SIAM 2025}, with the initial conditions \eqref{new ini} and \eqref{t1}, we also have that
\begin{equation}\label{lin h H2}
	\begin{cases}
		h>0,\ \phi>0,\ h\leq C_\eta,\\
		\Vert \phi \Vert_{ L^\infty}\leq C,\\
		\Vert r h_r \Vert_{ H^1}+\Vert r\phi_r \Vert_{ H^1}\leq C,\\
		\Vert r \phi_t \Vert_{L^2} \leq Cc_1.
	\end{cases}
\end{equation}

With $h>0$, we can multiply \eqref{lin sys}$_2$ by $-h^{-2}$, then we have
\begin{equation*}\label{lin equ of 1/h}
\left(\frac{1}{h}\right)_t+\left(\frac{1}{h}\right)_rU+(1-\delta)\frac{1}{h}U_r+\frac{2}{r}(1-\delta)\frac{1}{h}U=0,
\end{equation*}
similar to \eqref{lin sys}$_1$ and \eqref{lin rho H2}, we have
$$
 \left\Vert \frac{1}{h}\right\Vert_{ L^\infty}\leq \left\Vert \frac{1}{\tilde h_0} \exp\left((\delta-1)\int_0^t\left( U_r(r(\tau),\tau)+\frac{2}{r(\tau)}U(r(\tau),\tau)\right) d\tau \right)\right\Vert_{ L^\infty}\leq C,
$$ which yields that
\begin{equation}\label{lin ht/h}
	\begin{cases}
\displaystyle			\left\Vert r \frac{h_t}{h} \right\Vert_{ L^2}\leq C\left\Vert \frac{1}{h}\right\Vert_{ L^\infty}\Vert r h_r \Vert_{ L^2}\Vert U\Vert_{ L^\infty}+C\Vert r U_r \Vert_{ L^2}+C\Vert r U \Vert_{ L^2}\leq Cc_1,\\
\displaystyle			\left\Vert  \frac{h_t}{h}\right\Vert_{ L^\infty}\leq C\Vert \frac{1}{h}\Vert_{ L^\infty}\Vert h_r\Vert_{ L^\infty}\Vert U\Vert_{ L^\infty}+C\Vert U_r \Vert_{ L^\infty}+C\Vert U\Vert_{ L^\infty}\leq Cc_2,
	\end{cases}
\end{equation} for all $0\leq t \leq T \leq (Cc_2)^{-2}$.

Next, before we estimate $\Vert r h_{rt}\Vert_{ L^2}$, with \eqref{lin sys}$_2$ we have
$$
\left(\frac{h}{H}\right)_t=\frac{h_tH-H_th}{H^2}=\frac{h_t}{H}-\frac{H_t}{H}\frac{h}{H}=-\frac{h_r}{h}\frac{h}{H}U-(\delta-1)\frac{h}{H}\left(U_r+\frac{2}{r}U\right)-\frac{H_t}{H}\frac{h}{H},
$$ which yields that
\begin{equation}\label{equ of h/H}
	\left(\frac{h}{H}\right)_t+\frac{h}{H}\left(    \frac{h_r}{h}U+(\delta-1)(U_r+\frac{2}{r}U)+\frac{H_t}{H}   \right)=0,
\end{equation}and then, for all $0\leq t \leq T \leq (Cc_3)^{-2}$, we have
\begin{equation}\label{lin Linfty h/H}
	\left\Vert \frac{h}{H}\right\Vert_{ L^\infty}\leq C\left\Vert \frac{\tilde h_0}{H_0} \right\Vert_{ L^\infty}\exp\left( -\int_0^t \left(   \frac{1}{h} h_r U+(\delta-1)\left(U_r+\frac{2}{r}U\right)+\frac{H_t}{H}  \right)(s)ds \right)\leq C,
\end{equation} where we have used \eqref{co} and \eqref{c1,c2}. 

Then \eqref{c1,c2} and \eqref{lin Linfty h/H} yield that
\begin{equation}\label{lin hrt}
	\begin{aligned}
		\Vert r h_{rt}\Vert_{ L^2}\leq& C\Vert rh_{rr}\Vert_{ L^2}\Vert U\Vert_{ L^\infty}+\Vert h_r \Vert_{ L^\infty}\Vert r U_r \Vert_{ L^2}+C\Vert h_r \Vert_{ L^\infty}\Vert r U \Vert_{ L^2}\\
		&+C\left\Vert \frac{h}{H}\right\Vert_{ L^\infty}\left\Vert r H\left(U_r+\frac{2}{r}U\right)_r\right\Vert_{ L^2} \leq Cc_2.
	\end{aligned}
\end{equation}
Similarly, we have that $\Vert r \phi_{rt}\Vert_{ L^2}\leq Cc_2.$ Moreover, by virtue of \eqref{lin sys}$_{2,3}$, we obtain that
\begin{equation*}
	\begin{aligned}
		&(h \phi)_t=h_t\phi+h\phi_t=-\phi h_rU-(\delta-1)h \phi U_r-\frac{2}{r}(\delta-1)h \phi U\\
		&-h \phi_r U-(\gamma-\delta)h\phi U_r -\frac{2}{r}(\gamma-\delta) h\phi U,\\
	\end{aligned}
\end{equation*} i.e.
\begin{equation}\label{lin hphi-1}
 (h\phi)_t+(h \phi)_r U+(\gamma-1)(h \phi)\left(U_r+\frac{2}{r}U\right)=0,
\end{equation} together with \eqref{co} and \eqref{c1,c2}, we have
$$
\Vert h \phi \Vert_{ L^\infty}+\Vert r(h\phi)_r\Vert_{L^2}\leq C,
$$ for all $0\leq t \leq T \leq (Cc_3)^{-2}$. 

Then the proof of Lemma \ref{le:local rho W1p} is complete.
\end{proof}

\begin{lemma}\label{le:local v}
	Let $(\rho, h,\phi, v, u)$ be the solution to (\ref{lin sys}) and (\ref{lin initial}) on $I_a \times [0,T]$ which is stated in Proposition \ref{prop1} and $U$ satisfies (\ref{c1,c2}). Then the following estimates hold
	\begin{equation}\label{local v}
		\begin{cases}
			\Vert rv(\cdot,t) \Vert_{H^1}\leq C,\\
			\Vert rv(\cdot,t)_t \Vert_{L^2}\leq Cc_1,\\
		\end{cases}
	\end{equation}for all \begin{equation}\label{t2}
0\leq t \leq T\leq (Cc_3)^{-2}. 
	\end{equation}
\end{lemma}
\begin{proof}
Multiplying \eqref{lin sys}$_4$ by $r^2v$ and integrating by parts over $I_a$, we have
\begin{equation}\label{lin v L2}
	\begin{aligned}
		\frac{d}{dt}\Vert r v \Vert_{ L^2}^2\leq& C(\Vert U\Vert_{ L^\infty}+\Vert U_r\Vert_{ L^\infty}) \Vert r v \Vert_{ L^2}^2+C\Vert \phi \Vert_{ L^\infty}(\Vert rv \Vert_{ L^2}+\Vert r U \Vert_{ L^2})\Vert rv \Vert_{ L^2}\\
		\leq& Cc_2\Vert rv \Vert_{ L^2}^2+Cc_1^2.
	\end{aligned}
\end{equation} \eqref{t2}, \eqref{lin v L2} and Gronwall's inequality yield that
$$
\Vert rv \Vert_{ L^2}^2\leq C.
$$ 

Next, differentiating \eqref{lin sys}$_4$ with respect to $r$, we obtain that
\begin{equation}\label{lin equ vr}
	v_{rt}+U_rv_r+Uv_{rr}+C\phi_r(v-U)+C\phi(v_r-U_r)=0,
\end{equation} and then, multiplying \eqref{lin equ vr} by $r^2v_r$ and integrating by parts over $I_a$, we have
\begin{equation}\label{lin vrL2}
	\begin{aligned}
		\frac{d}{dt}\Vert rv_r \Vert_{ L^2}^2\leq& C\Vert U_r \Vert_{ L^\infty}\Vert r v_r \Vert_{ L^2}^2+C\Vert U \Vert_{ L^\infty} \Vert rv_r \Vert_{ L^2}^2\\
		&+C\Vert \phi_r \Vert_{ L^\infty}\Vert r(v-U)\Vert_{ L^2}\Vert rv_r\Vert_{ L^2}+C\Vert \phi \Vert_{ L^\infty}(\Vert rv_r \Vert_{ L^2}  +\Vert rU_r \Vert_{ L^2} )\Vert r v_r \Vert_{ L^2}\\
		\leq& Cc_2\Vert rv_r \Vert_{ L^2}^2+Cc_1^2,
	\end{aligned}
\end{equation} \eqref{t2}, \eqref{lin vrL2} and Gronwall's inequality yield that
$$
\Vert rv_r \Vert_{ L^2}^2\leq C.
$$ 

Finally, with \eqref{lin sys}$_4$, we have
$$
\Vert r v_t \Vert_{ L^2}\leq C\Vert U \Vert_{ L^\infty}\Vert r v_r \Vert_{ L^2}+C\Vert \phi \Vert_{ L^\infty}\Vert r(v-U)\Vert_{ L^2}\leq Cc_1.
$$
 
This completes the proof of Lemma \ref{le:local v}.
\end{proof}

\begin{lemma}\label{le:local u}
Let $(\rho, h,\phi, v, u)$ be the solution to (\ref{lin sys}) and (\ref{lin initial}) on $I_a \times [0,T]$ which is stated in Proposition \ref{prop1} and $(U,H)$ satisfies (\ref{c1,c2}). Then the following estimates hold
\begin{equation}\label{local u}
\begin{cases}
\displaystyle\Vert ru(\cdot,t) \Vert_{H^1}+\Vert ru(\cdot,t)_t \Vert_{L^2}+\int_0^t (\Vert r u_{rt}\Vert_{ L^2}^2+ \Vert r u_{rrr}\Vert_{ L^2}^2)(s)ds \leq C,\\
\displaystyle \Vert rh(u_r+\frac{2}{r}u)_r(\cdot,t) \Vert_{ L^2}+
\Vert r  u_{rr}(\cdot,t)\Vert_{ L^2} \leq Cc_1^2,
\end{cases}
\end{equation}for all 
\begin{equation}\label{t3}
0\leq t \leq T\leq (Cc_3)^{-5}.
\end{equation} 
\end{lemma}
\begin{proof}
The proof is divided into three steps.\\

\textbf{Step 1: the estimate on $\Vert r u \Vert_{L^2}$.}

 Multiplying  $(\ref{lin sys})_5$ by $2r^2u$, and integrating the result by parts over $I_a$, with the following fact
 \begin{equation*}
 	\begin{aligned}
 		&2\delta \int r^2 h u_{rr} u+2\delta \int  r^2 h(\frac{2}{r}u)_ru\\
 		=&-2\delta \int 2r h u_r u -2\delta \int r^2 h_r u_r u -2\delta \int r^2 h u_r^2+2\delta \int 2r h u_r u -4\delta  \int h u^2\\
 		=& -2\delta \int r^2 h_r u_r u -2\delta \int r^2 h u_r^2 -4\delta  \int h u^2,
 	\end{aligned}
 \end{equation*} we have
\begin{equation}\label{lin uL2}
	\begin{split}
		&\frac{d}{dt}\Vert ru \Vert_{L^2}^2+2\delta\int r^2 h u_r^2+4\delta\int hu^2=-2\int r^2 UU_r u+\frac{\gamma}{\delta}\int r^2 \phi (v-U) u-2\delta\int r^2 h_r u_r u\\
		&-2\delta\int r^2 (v-U) u_ru +4(\delta-1)\int r (v-U) Uu\\
		\leq&C\Vert U \Vert_{ L^\infty}\Vert rU_r\Vert_{ L^2}\Vert ru\Vert_{ L^2}+C\Vert \phi \Vert_{ L^\infty}\Vert r(v-U) \Vert_{ L^2}\Vert r u \Vert_{ L^2}+C\Vert \frac{1}{h} \Vert_{ L^\infty}\Vert h_r \Vert_{ L^\infty}^2 \Vert r u\Vert_{ L^2}^2\\
		&+C\Vert \frac{1}{h}\Vert_{ L^\infty}\Vert v-U \Vert_{ L^\infty}^2\Vert ru \Vert_{ L^2}^2+C\Vert (v-U)\Vert_{ L^\infty}\Vert r U \Vert_{ L^2}  \Vert r u \Vert_{ L^2}  +\varepsilon \int r^2 h u_r^2\\
		\leq& Cc_1^2\Vert ru \Vert_{ L^2}^2+Cc_1^4+\varepsilon \int r^2 h u_r^2,
	\end{split}
\end{equation} where we have used (\ref{c1,c2}), Lemma \ref{le:local rho W1p}, Young's inequality.

Together with \eqref{t3}, \eqref{lin uL2} and Gronwall's inequality, we obtain that
$$
\Vert r u \Vert_{ L^2}^2 \leq C.
$$

\textbf{Step 2: the estimate on $\Vert r u_r \Vert_{L^2}$.}\\
Multiplying \eqref{lin sys}$_5$ by $r^2\frac{1}{h}u_t$, and then integrating over $I_a$, with the fact that
\begin{equation*}
	\begin{aligned}
		&\delta \int h(u_r+\frac{2}{r}u)_r r^2 \frac{1}{h}u_t =\delta\int u_{rr}  r^2 u_t+\delta \int  \frac{2}{r}u_r r^2 u_t-2\delta\int uu_t\\
		=&-\delta\int u_{r}  r^2 u_{rt}-2\delta\int uu_t=-\frac{\delta}{2}\frac{d}{dt}\Vert r u_r \Vert_{ L^2}^2-\delta\frac{d}{dt} \Vert  u \Vert_{ L^2}^2.
	\end{aligned}
\end{equation*} we have
\begin{equation}\label{lin ur L2}
   \begin{aligned}
   	&\frac{\delta}{2}\frac{d}{dt}\Vert r u_r \Vert_{ L^2}^2+\delta\frac{d}{dt} \Vert  u \Vert_{ L^2}^2+\int r^2 \frac{1}{h}u_t^2\\
   	=&-\int r^2\frac{1}{h}u_t UU_r-\frac{\gamma}{2\delta}\int r^2\frac{1}{h}u_t \phi (v-U)-\delta\int r^2\frac{1}{h}u_t (v-U)u_r +2(\delta-1)\int r\frac{1}{h}u_t (v-U) U\\
   	\leq& \varepsilon \int r^2 \frac{1}{h} u_t^2+C\Vert r^2 \frac{1}{h} U^2 U_r^2\Vert_{ L^1}+C\Vert r^2 \frac{1}{h} \phi^2 (v-U)^2\Vert_{ L^1}+C\Vert r^2 \frac{1}{h} (v-U)^2 u_r^2\Vert_{ L^1}+C\Vert  \frac{1}{h} (v-U)^2 U^2\Vert_{ L^1}\\
   	\leq&\varepsilon \int r^2 \frac{1}{h} u_t^2+Cc_1^4+Cc_1^2\Vert r u_r \Vert_{ L^2}^2,
   \end{aligned}
\end{equation} where we have used Lemma \ref{le:local rho W1p}.

Together with \eqref{t3}, \eqref{lin ur L2} and Gronwall's inequality, we have
\begin{equation}\label{lin ur L2 finnal}
	\Vert ru_r \Vert_{ L^2}^2\leq C.
\end{equation}

\textbf{Step 3: the estimate on $\Vert r u_t \Vert_{L^2}$.}

Firstly, with \eqref{lin sys}$_5$, \eqref{lin ur L2 finnal} and Lemma \ref{le:local rho W1p}, we have
\begin{equation}\label{lin hurr}
	\begin{cases}
		\Vert r h(u_r +\frac{2}{r}u)_r \Vert_{ L^2}\\
		\leq C\Vert r u_t \Vert_{ L^2}+C\Vert U \Vert_{ L^\infty}\Vert r U_r \Vert_{ L^2}+C\Vert \phi \Vert_{ L^\infty}\Vert r (v-U) \Vert_{ L^2}
		+C\Vert (v-U) \Vert_{ L^\infty}(\Vert r u_r \Vert_{ L^2}+\Vert rU \Vert_{ L^2})\\
		\leq C\Vert r u_t \Vert_{ L^2}+Cc_1^2,\\[2mm]
		\Vert r u_{rr}\Vert_{ L^2}\\
		=\Vert r \frac{1}{h} \cdot h(u_r +\frac{2}{r}u-\frac{2}{r}u)_r \Vert_{ L^2}\leq C\Vert  r  h(u_r +\frac{2}{r}u)_r \Vert_{ L^2}+C\Vert r u_r \Vert_{ L^2}+C\Vert r u \Vert_{ L^2}\leq C\Vert r u_t \Vert_{ L^2}+Cc_1^2,\\[2mm]
		\Vert u_r \Vert_{ L^\infty}\leq C\Vert ru_r \Vert_{ L^2}+C\Vert r u_{rr}\Vert_{ L^2}\leq C\Vert r u_t \Vert_{ L^2}+Cc_1^2,
	\end{cases}
\end{equation} for all $0\leq t \leq T \leq (Cc_3)^{-4}$.

Secondly, differentiating \eqref{lin sys}$_5$ with respect to $t$, we obtain that
\begin{equation}\label{lin utt}
	\begin{aligned}
			&u_{tt}+(UU_r)_t+\frac{\gamma}{2\delta}\phi_t (v-U)+\frac{\gamma}{2\delta}\phi (v-U)_t=\delta h_t(u_r+\frac{2}{r}u)_r+\delta h(u_r+\frac{2}{r}u)_{rt}\\
		&+\delta (v-U)_t u_r+\delta (v-U)u_{rt}+(\delta-1) \frac{2}{r} (v-U)_tU+(\delta-1) \frac{2}{r} (v-U)U_t,
	\end{aligned}
\end{equation}
and multplying \eqref{lin utt} by $2r^2u_t$, and then integrating by parts over $I_a$, with the fact
\begin{equation}
	\begin{aligned}
		&2\delta\int r^2 h(u_r+\frac{2}{r}u)_{rt} u_t\\
		=&-2\delta \int 2r hu_{rt} u_t-2\delta\int r^2 h_r u_{rt} u_t-2\delta\int r^2 hu_{rt}^2+2\delta\int 2r h u_{rt} u_t-4\delta\int  hu_t u_t\\
		=&-2\delta\int r^2 h_r u_{rt} u_t-2\delta\int r^2 hu_{rt}^2-4\delta\int  hu_t^2.\\
	\end{aligned}
\end{equation} we arrive that
\begin{equation}\label{lin ut L2}
	\begin{aligned}
	&\frac{d}{dt}\Vert r u_t \Vert_{ L^2}^2+2\delta\int r^2 hu_{rt}^2+4\delta\int  hu_t^2\\
	=& 2\int r^2u_t \left( -(\frac{U^2}{2})_{rt}-\frac{\gamma}{2\delta}\phi_t (v-U)-\frac{\gamma}{2\delta}\phi (v-U)_t+\delta h_t(u_r+\frac{2}{r}u)_r \right) \\
	&+2\int r^2 u_t \left(\delta (v-U)_t u_r+\delta (v-U)u_{rt}+(\delta-1) \frac{2}{r} (v-U)_tU+(\delta-1) \frac{2}{r} (v-U)U_t \right)-2\delta\int r^2 h_r u_{rt} u_t\\
	\leq& C\Vert U \Vert_{ L^\infty}\Vert rU_t \Vert_{ L^2}\Vert r u_t \Vert_{ L^2}+C\Vert \frac{1}{h}\Vert_{ L^\infty}\Vert U \Vert_{ L^\infty}^2\Vert rU_t\Vert_{ L^2}^2+\varepsilon \int r^2hu_{rt}^2+C\Vert r \phi_t \Vert_{ L^2}\Vert (v-U) \Vert_{ L^\infty}\Vert r u_t \Vert_{ L^2}\\
	&+C\Vert \phi \Vert_{ L^\infty} \Vert r (v-U)_t\Vert_{ L^2}\Vert r u_t \Vert_{ L^2}+ C\Vert \frac{h_t}{h}\Vert_{ L^\infty}\Vert rh(u_r+\frac{2}{r}u)_r\Vert_{ L^2}\Vert r u_t\Vert_{ L^2}+\varepsilon \int r^2hu_{rt}^2 \\
	&+C\Vert \frac{1}{h}\Vert_{ L^\infty}\Vert h_r \Vert_{ L^\infty}^2\Vert r u_t \Vert_{ L^2}^2 
	+C\Vert r(v-U)_t\Vert_{ L^2}\Vert u_r \Vert_{ L^\infty}\Vert r u_t \Vert_{ L^2}+C\Vert \frac{1}{h}\Vert_{ L^\infty}\Vert r (v-U) \Vert_{ L^\infty}^2\Vert r u_t \Vert_{ L^2}^2 \\
	&+\varepsilon \int r^2hu_{rt}^2+ C\Vert r (v-U)_t\Vert_{ L^2}\Vert U \Vert_{ L^\infty}\Vert r u_t \Vert_{ L^2}
	+C\Vert  (v-U) \Vert_{ L^\infty}\Vert r U_t \Vert_{ L^2}\Vert r u_t \Vert_{ L^2} \\
	\leq&Cc_2^2\Vert ru_t \Vert_{ L^2}^2+Cc_2\Vert u_r\Vert_{ L^\infty}^2+C\Vert rh(u_r+\frac{2}{r}u)_r\Vert_{ L^2}^2+Cc_2^4+3\varepsilon\int r^2 h u_{rt}^2\\
	\leq&Cc_2^2\Vert ru_t \Vert_{ L^2}^2+Cc_2^5+3\varepsilon\int r^2 h u_{rt}^2,
	\end{aligned}
\end{equation} where we used Lemma \ref{le:local rho W1p}, \eqref{c1,c2}, \eqref{lin hurr}. 

\eqref{lin ut L2}, \eqref{t3} combining with the compatibility conditions \eqref{compatibility condition local}, we obtain that
$$
\Vert ru_t \Vert_{ L^2} +\int_0^t \Vert r  \sqrt{h} u_{rt}\Vert_{ L^2}^2 \leq C,
$$ which together with \eqref{lin hurr} mean that
\begin{equation}\label{lin hurr finnal}
	\begin{cases}
\displaystyle		\Vert rh(u_r+\frac{2}{r}u)_r \Vert_{ L^2}\leq Cc_1^2,\\
\displaystyle		\Vert r  u_{rr}\Vert_{ L^2} \leq Cc_1^2,\\
\displaystyle		\int_0^t \Vert r   u_{rt}\Vert_{ L^2}^2\leq C.
	\end{cases}
\end{equation}

Finally, by virtue of \eqref{lin sys}$_5$ and $h>0$, we have
\begin{equation}\label{lin urrr}
u_{rrr}\sim (\frac{1}{h} u_t)_r+(\frac{1}{h}  UU_r)_r  +(\frac{1}{h} \phi (v-U))_r+(\frac{2}{r}u)_{rr}+(\frac{1}{h} (v-U) u_r)_r+(\frac{1}{h}  (v-U) \frac{2}{r}U)_r,
\end{equation} and then, we obtain that
\begin{equation}\label{lin urrr-2}
	\begin{aligned}
		\int_0^t \Vert r u_{rrr}\Vert_{ L^2}^2(s) ds\leq \int_0^t (Cc_2^4+C\Vert r u_{rt}\Vert_{ L^2}^2)(s)ds \leq C,
	\end{aligned}
\end{equation} for all $0\leq t \leq T \leq (Cc_3)^{-5}$.

Then, the proof of Lemma \ref{le:local u} is complete.

\end{proof}

 \subsection*{Proof of Proposition \ref{prop2}}
Proposition \ref{prop2} can be given by Lemmas \ref{le:local rho W1p} and \ref{le:local u}.
In fact, choosing $(c_1, c_2, c_3)$ sufficiently large such that
\begin{equation}\label{def ci}
\begin{cases}
c_1\ge C,\\
c_2\ge C+Cc_1+Cc_1^2,\\
c_3\ge Cc_2,\\
T\leq (Cc_3)^{-5},
\end{cases}
\end{equation}then we prove Proposition \ref{prop2} completely.

\subsection{An iteration scheme and strong convergence with $\eta>0$}\label{sec2.3}
Consider the following approximation system:
\begin{equation}\label{approximate sys}
	\left\{
\begin{aligned}
	&\rho_t^k+\rho_r^k u^{k-1}+\rho^k u^{k-1}_r+\frac{2}{r}\rho^k u^{k-1}=0,\\
	&h_t^k+h_r^ku^{k-1}+(\delta-1)h^ku_r^{k-1}+\frac{2}{r}(\delta-1)h^ku^{k-1}=0,\\
	&\phi^k_t+\phi^k_ru^{k-1}+(\gamma-\delta)\phi^k u^{k-1}_r+\frac{2}{r}(\gamma-\delta)\phi^k u^{k-1}=0,\\
	&v^k_t+u^{k-1}v^k_r+\frac{\gamma}{2\delta}\phi^k(v^k-u^{k-1})=0,\\
	&u_t^k+u^{k-1}u_r^{k-1}+\frac{\gamma}{2\delta}\phi^k(v^k-u^{k-1})=\delta h^k(u_r^k+\frac{2}{r}u^k)_r+\delta (v^k-u^{k-1})u^k_r+(\delta-1) (v^k-u^{k-1})\frac{2}{r}u^{k-1},
\end{aligned}
\right.
\end{equation} (\ref{approximate sys}) is equipped with the following condition:
\begin{equation}\label{approximate initial}
\begin{cases}
(\rho^k, h^k, \phi^k, v^k, u^k)(r,0)=(\tilde \rho_0, \tilde h_0, \tilde \phi_0, \tilde{v}_0, u_0),\ \ r\in I_a,\\
u^k(r,t)|_{r=a}=0,\ \ t\ge0,\\
u^k \to 0   \ \ \text{as}\ \  r \to +\infty,\\
\end{cases}
\end{equation}
for $k=1,2,3,\cdot\cdot\cdot$, where $(\rho^0,h^0,\phi^0,v^0, u^0)=(\tilde \rho_0, \tilde h_0, \tilde \phi_0, \tilde v_0, u_0)$.

With Subsection \ref{sec 2.3}, we know that ($\rho^k$, $h^k$, $\phi^k$, $v^k$, $u^k$) enjoys the following estimates 

\begin{equation}\label{lin k-estimates-1}
	\begin{cases}
\displaystyle	\Vert r u^k(\cdot,t) \Vert_{H^1} \leq c_1,\ \ 
\displaystyle \Vert ru^k_t (\cdot,t)\Vert_{ L^2}+	\Vert r u^k(\cdot,t) \Vert_{D^2} +\int_0^T (\Vert r u^k_{rrr}\Vert_{ L^2}^2+\Vert r u^k_{rt}\Vert_{ L^2}^2)(\cdot,s)ds \leq c_2,\\
\displaystyle	\Vert rh^k(\cdot,t)_r \Vert_{H^1}+\Vert r 
	\frac{h^k_t}{h^k}(\cdot,t) \Vert_{L^2 }+\Vert r h^k(u^k_r+\frac{2}{r}u^k)_r(\cdot,t)\Vert_{L^2}\leq c_2,\\
\displaystyle	\Vert rh^k(\cdot,t)_{rt} \Vert_{L^2}+\Vert 
	\frac{h^k_t}{h^k}(\cdot,t) \Vert_{L^\infty}\leq c_3,
	\end{cases}
\end{equation} and
\begin{equation}\label{lin k-estimates-2}
	\begin{cases}
\displaystyle\Vert \phi^k (\cdot,t)\Vert_{ L^\infty}+\Vert r(\rho^k-\eta)(\cdot,t)\Vert_{ H^2}+\Vert r \phi^k_r (\cdot,t) \Vert_{ H^1}+\Vert \frac{h^k}{h^{k-1}}\Vert_{ L^\infty}\\
\displaystyle+\Vert r v^k(\cdot,t)\Vert_{ H^1}+\Vert h^k \phi^k \Vert_{ L^\infty} +\Vert r(h^k\phi^k)_r\Vert_{ L^2}\leq C\\
\displaystyle\Vert r v_t^k(\cdot,t) \Vert_{ L^2}+\Vert r\phi^k_t(\cdot,t) \Vert_{ L^2}+\Vert r\rho^k_t(\cdot,t)\Vert_{ L^2}\leq Cc_1,\\
\displaystyle\Vert r \rho^k_{rt}(\cdot,t)\Vert_{ L^2}+\Vert r \phi^k_{rt}(\cdot,t)\Vert_{ L^2}\leq Cc_2,\\
\displaystyle\Vert h^k (\cdot,t)\Vert_{ L^\infty}+\Vert r h^k_t(\cdot,t)\Vert_{ L^2}\leq C_\eta,
	\end{cases}
\end{equation} for all $0\leq t \leq T\leq (Cc_3)^{-5}$.

Then there exist a subsequence ($\rho^{k_i}$, $h^{k_i}$, $\phi^{k_i}$, $v^{k_i}$, $u^{k_i}$) and $r(\rho-\eta) \in L^\infty(0,T;H^2)$, $r\phi\in L^\infty(0,T;D^1 \cap D^2)$, $r h \in L^\infty(0,T;D^1\cap D^2)$,  $rv \in L^\infty(0,T;H^1)$, $ru \in L^\infty(0,T;H^2)\cap L^2(0,T;D^3)$ such that
\begin{equation}\label{weak conv}
\begin{split}
r(\rho^{k_i}-\eta)\rightharpoonup r(\rho-\eta)\ \ &\text{weakly*}\ \ \text{in}\ \  L^\infty(0,T;H^2),\\
(r\phi^{k_i},rh^{k_i})\rightharpoonup (r\phi,rh)  \ \ &\text{weakly*}\ \ \text{in}\ \  L^\infty(0,T;D^1 \cap D^2),\\
rv^{k_i}\rightharpoonup rv \ \ &\text{weakly*}\ \ \text{in}\ \  L^\infty(0,T;H^1),\\
ru^{k_i}\rightharpoonup ru \ \ &\text{weakly*}\ \ \text{in}\ \  L^\infty(0,T;H^2)\cap L^2(0,T;D^3),\\
(r\rho^{k_i}_t,rh^{k_i}_t,r\phi^{k_i}_t,rv^{k_i}_t)\rightharpoonup (r\rho_t,rh_t,r\phi_t, rv_t)\ \ &\text{weakly*}\ \ \text{in}\ \  L^\infty(0,T;L^2),\\
ru_t^{k_i}\rightharpoonup ru_t \ \ &\text{weakly*}\ \ \text{in}\ \  L^\infty(0,T;L^2)\cap L^2(0,T;D^1),\\
\end{split}
\end{equation} and ($\rho,h,\phi,v,u$) satisfies \eqref{lin k-estimates-1} and \eqref{lin k-estimates-2}.


To justify that the limit of ($\rho^{k_i}$, $h^{k_i}$, $\phi^{k_i}$, $v^{k_i}$, $u^{k_i}$) is the same as that of their neighbor subsequence ($\rho^{k_i-1}$, $h^{k_i-1}$, $\phi^{k_i-1}$, $v^{k_i-1}$, $u^{k_i-1}$) which appears in (\ref{approximate sys}), we are going to prove that the full sequence ($\rho^k$, $h^k$, $\phi^k$, $v^k$, $u^k$) converges strongly to the limit ($\rho$, $h$, $\phi$, $v$, $u$), which is the desired strong solution. Denote $\bar \rho^k=\rho^k-\rho^{k-1}$, $\bar h^k=h^k-h^{k-1}$, $\bar \phi^k=\phi^k-\phi^{k-1}$, $\bar v^k=v^k-v^{k-1}$, $\bar u^k=u^k-u^{k-1}$ and $\overline{h^k\phi^k}=h^k\phi^k-h^{k-1}\phi^{k-1}$ for $k\ge2$. 

In this subsection, we let $Cc_i=C$, and in order to obtain uniform estimates independent of $\eta$, we can not utilize the $\Vert h^k \Vert_{ L^\infty}$ and $\Vert rh^k_t \Vert_{ L^2}$ in \eqref{lin k-estimates-2}$_4$.

\begin{lemma}\label{conv of rho}
The following estimate:
\begin{equation}\label{bar rho L2}
\frac{d}{dt}(\Vert \bar \rho^k (t)\Vert_{L^2}^2+\Vert \bar \phi^k (t)\Vert_{L^2}^2+\Vert \bar v^k(t)\Vert_{ L^2}^2)\leq C(\Vert \bar \rho^k(t)\Vert_{L^2}^2 +\Vert \bar \phi^k(t)\Vert_{L^2}^2+C\Vert \bar v^k(t)\Vert_{ L^2}^2)+ \varepsilon\Vert \bar u^{k-1}(t)\Vert_{H^1}^2
\end{equation}
holds for any $0\leq t \leq T\le (Cc_3)^{-5}$.
\end{lemma}
\begin{proof}We obtain from (\ref{approximate sys})$_1$ that
\begin{equation}\label{bar rho}
\bar\rho^k_t+\bar \rho^k_r u^{k-1}+ \rho^{k-1}_r \bar u^{k-1}+\bar \rho^k u_r^{k-1}+\rho^{k-1} \bar u_r^{k-1}+\frac{2}{r}\bar \rho^k u^{k-1}+\frac{2}{r} \rho^{k-1} \bar u^{k-1}=0
\end{equation}

Multiplying $(\ref{bar rho})$ by $\bar \rho^k$, integrating the result by parts over $I_a$, we have
\begin{equation*}
\begin{aligned}
\frac{d}{dt}\Vert \bar \rho^{k}\Vert_{L^2}^2\leq& C\Vert \bar \rho^k\Vert_{ L^2}^2\Vert u^{k-1}_r\Vert_{ L^\infty}+C\Vert \rho_r^{k-1}\Vert_{ L^\infty}\Vert \bar \rho^k \Vert_{ L^2}\Vert \bar u^{k-1}\Vert_{L^2}\\
&+C\Vert \rho^{k-1}\Vert_{ L^\infty}\Vert \bar \rho^k\Vert_{ L^2}\Vert \bar u^{k-1}_r\Vert_{ L^2}+C\Vert \bar \rho^k\Vert_{ L^2}^2\Vert u^{k-1} \Vert_{ L^\infty}+C\Vert \rho^{k-1}\Vert_{ L^\infty}\Vert \bar \rho^k \Vert_{ L^2}\Vert \bar u^{k-1}\Vert_{ L^2}\\
\leq& C\Vert \bar \rho^k \Vert_{ L^2}^2+\varepsilon\Vert \bar u^{k-1} \Vert_{ H^1}^2.
\end{aligned}
\end{equation*} 

Similarly, we have
\begin{equation*}
	\begin{cases}
\displaystyle		\frac{d}{dt}\Vert \bar \phi^{k}\Vert_{L^2}^2\leq C\Vert \bar \phi^k \Vert_{ L^2}^2+\varepsilon\Vert \bar u^{k-1} \Vert_{ H^1}^2,\\[3mm]
\displaystyle		\frac{d}{dt}\Vert \bar v^k \Vert_{L^2}^2\leq C\Vert \bar v^k \Vert_{ L^2}^2+C\Vert \bar \phi^k \Vert_{ L^2}^2+\varepsilon \Vert \bar u^{k-1}\Vert_{ H^1}^2.\\
	\end{cases}
\end{equation*}

Then we complete the proof of Lemma \ref{conv of rho}. 
\end{proof}

\begin{lemma}\label{conv of h}
The following estimate:
\begin{equation}
	\begin{aligned}
&\frac{d}{dt}\Vert \bar h^k(t) \Vert_{L^2}^2\leq C \Vert \bar h^k(t) \Vert_{L^{2}}^2+\varepsilon(\Vert \bar u^{k-1}(t)\Vert_{H^1}^2+  \Vert h^{k-1}\bar u^{k-1}_r(t) \Vert_{ L^2}^2+    \Vert h^{k-1}\frac{2}{r}\bar u^{k-1}(t)\Vert_{ L^2}^2      )
\end{aligned}
\end{equation}
holds for any $0\leq t \leq T\le (Cc_3)^{-5}$.
\end{lemma}
\begin{proof}
We obtain from (\ref{approximate sys})$_2$ that
\begin{equation}\label{equ of bar h}
	\begin{aligned}
&\bar h_t^k+\bar h^k_r u^{k-1}+h^{k-1}_r\bar u^{k-1}
+(\delta-1)\left( \bar h^k u^{k-1}_r+h^{k-1}\bar u^{k-1}_r \right)+\frac{2(\delta-1)}{r}\left(  \bar h^k u^{k-1}+h^{k-1}\bar u^{k-1}  \right)=0
	\end{aligned}
\end{equation}

Multiplying (\ref{equ of bar h}) by $\bar h^k$ and integrating the result over $I_a$, we have
\begin{equation*}
\begin{split}
\frac{d}{dt}\Vert \bar h^k \Vert_{L^{2}}^2\leq&C\Vert \bar h^k \Vert_{ L^2}^2\Vert u^{k-1}_r\Vert_{ L^\infty}+C\Vert r h^{k-1}_r \Vert_{ L^2}\Vert \bar u^{k-1}\Vert_{ L^\infty}\Vert \bar h^k \Vert_{ L^2}
+C\Vert h^{k-1}\bar u^{k-1}_r \Vert_{ L^2}\Vert \bar h^k \Vert_{ L^2}\\
&+C\Vert \bar h^k \Vert_{ L^2}^2 \Vert u^{k-1}\Vert_{ L^\infty}+C\Vert \bar h^k\Vert_{ L^2}\Vert h^{k-1}\frac{2}{r}\bar u^{k-1}\Vert_{ L^2}
\\
\leq&C\Vert \bar h^k \Vert_{L^{2}}^2+\varepsilon(\Vert \bar u^{k-1}\Vert_{H^1}^2+  \Vert h^{k-1}\bar u^{k-1}_r \Vert_{ L^2}^2+    \Vert h^{k-1}\frac{2}{r}\bar u^{k-1}\Vert_{ L^2}^2      ),
\end{split}
\end{equation*}
where we have used Young's inequality and \eqref{lin k-estimates-1}.

The proof of Lemma $\ref{conv of h}$ is complete.
\end{proof}

\begin{lemma}\label{conv of hphi}
	The following estimate:
	\begin{equation}\label{bar hphi L2}
		\begin{split}
			&\frac{d}{dt}\Vert \overline {h^k\phi^k}(t) \Vert_{L^{2}}^2\leq C \Vert \overline{h^k\phi^k}(t)\Vert_{ L^2}^2+\varepsilon\Vert \bar u^{k-1}(t)\Vert_{ H^1}^2
		\end{split}
	\end{equation}
	holds for any $0\leq t \leq T\le (Cc_3)^{-5}$.
\end{lemma}
\begin{proof}
	Recalling \eqref{lin hphi-1}, we have
	$$
	(h^k\phi^k)_t+(h^k \phi^k)_r u^{k-1}+(\gamma-1)(h^k \phi^k)(u^{k-1}_r+\frac{2}{r}u^{k-1})=0,
	$$ similar to Lemma \ref{conv of rho}, we have
	$$
	\Vert \overline{h^k\phi^k} \Vert_{ L^2}\leq C \Vert \overline{h^k\phi^k}\Vert_{ L^2}+\varepsilon\Vert \bar u^{k-1}\Vert_{ H^1}.
	$$
	
	Then, we have proved this lemma.
\end{proof}

\begin{lemma}\label{conv of u}
The following estimate:
\begin{equation}\label{bar u L2}
\begin{split}
&\frac{d}{dt}\Vert\sqrt{h^k} \bar u^k(t) \Vert_{L^{2}}^2+\delta\Vert h^k\bar u^k_r(t)\Vert_{ L^2}^2+\frac{\delta}{4}\Vert h^k \frac{2}{r}\bar u^k (t)\Vert_{ L^2}^2\\
\leq&C\Vert \sqrt{h^k}\bar u^k(t) \Vert_{L^{2}}^2+C\Vert \bar h^k(t) \Vert_{ L^2}^2+C\Vert \overline{h^k\phi^k}(t)\Vert_{ L^2}^2+C\Vert \bar v^k(t)\Vert_{ L^2}^2\\
&+\varepsilon \Vert \sqrt{h^{k-1}}\bar u^{k-1}(t) \Vert_{L^{2}}^2+\varepsilon\Vert {h^{k-1}}\bar u^{k-1}_r(t)\Vert_{ L^2}^2+\varepsilon\Vert h^{k-1}\frac{2}{r}\bar u^{k-1}(t)\Vert_{ L^2}^2
\end{split}
\end{equation}
holds for any $0\leq t \leq T\le (Cc_3)^{-5}$.
\end{lemma}

\begin{proof}
By virtue of \eqref{lin k-estimates-1} and \eqref{lin k-estimates-2}, we have
\begin{equation}\label{bar ul22}
\begin{cases}
\displaystyle\Vert \bar u^k \Vert_{L^{2}}=\Vert  \sqrt{h^k} \frac{1}{\sqrt{h^k}}\bar u^k \Vert_{L^{2}}\leq \Vert \frac{1}{\sqrt{h^k}}\Vert_{ L^\infty} \Vert  \sqrt{h^k}\bar u^k \Vert_{L^{2}}\leq C\Vert  \sqrt{h^k}\bar u^k \Vert_{L^{2}},\\[2mm]
\displaystyle\Vert \bar u^k\Vert_{H^1}\leq \Vert \bar u^k \Vert_{ L^2}+\Vert \bar u_r^k \Vert_{ L^2}\leq C(\Vert \sqrt{h^k}\bar u^k \Vert_{ L^2}+\Vert h^k \bar u_r^k \Vert_{ L^2}),\\
\displaystyle| h^{k} |=| \frac{h^k}{h^{k-1}} h^{k-1}| \leq C\Vert \frac{h^k}{h^{k-1}} \Vert_{ L^\infty}|h^{k-1}| \leq C| h^{k-1}|.
\end{cases}
\end{equation}
 
Since $h^k>0$ and $\Vert h^k \phi^k \Vert_{ L^\infty}\leq C$, we first multiply both sides of (\ref{approximate sys})$_4$ by $h^k$, we have
 \begin{equation*}\label{hu}
 	\begin{aligned}
&u_t^kh^k+u^{k-1}u_r^{k-1}h^k+C (h^k\phi^k) (v^k-u^{k-1})\\
=&\delta (u_r^k+\frac{2}{r}u^k)_rh^kh^k+\delta (v^k-u^{k-1}) u^k_rh^k+(\delta-1) (v^k-u^{k-1})\frac{2}{r}u^{k-1}h^k,
 	\end{aligned}
 \end{equation*}
and then, we have
\begin{equation}\label{bar hu}
\begin{split}
&\bar u_t^kh^k+u_t^{k-1}\bar h^k+\bar u^{k-1}u_r^{k-1}h^k+u^{k-2}\bar u_r^{k-1}h^k+u^{k-2}u_r^{k-2}\bar h^k+C \overline{h^k\phi^k} (v^k-u^{k-1})\\
&+C (h^{k-1}\phi^{k-1}) (\bar v^k-\bar u^{k-1})\\
&=\delta (\bar u_r^k+\frac{2}{r}\bar u^k)_rh^kh^k+\delta (u_r^{k-1}+\frac{2}{r}u^{k-1})_r\bar h^kh^k+\delta (u_r^{k-1}+\frac{2}{r}u^{k-1})_rh^{k-1}\bar h^k\\
&+\delta (\bar v^k-\bar u^{k-1}) u^k_rh^k+\delta (v^{k-1}-u^{k-2}) \bar u^k_rh^k+\delta (v^{k-1}-u^{k-2}) u^{k-1}_r\bar h^k\\
&+(\delta-1) (\bar v^k-\bar u^{k-1}) \frac{2}{r}u^{k-1}h^k+(\delta-1) (v^{k-1}-u^{k-2})\frac{2}{r}\bar u^{k-1}h^k+(\delta-1) (v^{k-1}-u^{k-2})\frac{2}{r}u^{k-2}\bar h^k,
\end{split}
\end{equation} next, multiplying (\ref{bar hu}) by $2\bar u^k$ and integrating the result by parts over $I_a$, with \eqref{lin k-estimates-1}, \eqref{lin k-estimates-2} and \eqref{bar ul22}$_3$, we obtain that

\begin{equation}\label{bar hu L2}
\begin{split}
&\frac{d}{dt}\Vert \sqrt{h^k}\bar u^k \Vert_{L^{2}}^2\\
&\leq C\Vert \frac{h_t^k}{h^k}\Vert_{ L^\infty}\Vert  \sqrt{h^k}\bar u^k\Vert_{ L^2}^2\\&+C\Vert ru^{k-1}_t \Vert_{ L^2}\Vert \bar h^k \Vert_{ L^2}\Vert \bar u^k \Vert_{ L^\infty}+C\Vert \frac{\sqrt{h^k}}{\sqrt{h^{k-1}}}\Vert_{ L^\infty}\Vert \sqrt{h^{k-1}}\bar u^{k-1}\Vert_{ L^2}\Vert u^{k-1}_r\Vert_{ L^\infty}\Vert \sqrt{h^k}\bar u^k \Vert_{ L^2} \\
&+ C\Vert \frac{\sqrt{h^k}}{h^{k-1}}\Vert_{ L^\infty} \Vert u^{k-2}\Vert_{ L^\infty}\Vert h^{k-1}\bar u^{k-1}_r \Vert_{ L^2}\Vert \sqrt{h^k}\bar u^k \Vert_{ L^2}+C\Vert u^{k-2}\Vert_{ L^\infty}\Vert u^{k-2}_r \Vert_{ L^\infty}\Vert \bar h^k \Vert_{ L^2}\Vert \bar u^k \Vert_{ L^2} \\
&+C\Vert \overline{h^k\phi^k}\Vert_{ L^2}\Vert rv^k-ru^{k-1} \Vert_{ L^2}\Vert \bar u^k \Vert_{ L^\infty}+C\Vert h^{k-1}\phi^{k-1}\Vert_{ L^\infty}\Vert \bar v^{k}-\bar u^{k-1} \Vert_{ L^2}\Vert \bar u^k \Vert_{ L^2}+\sum_{i=1}^9 W_i\\
\leq&C\Vert \sqrt{h^k}\bar u^k \Vert_{L^{2}}^2+C\Vert \bar v^k \Vert_{ L^2}^2+C\Vert \bar h^k \Vert_{ L^2}^2+C\Vert \overline{h^k\phi^k}\Vert_{ L^2}^2+\frac{\delta}{5}\Vert \bar u^k \Vert_{ H^1}^2+\varepsilon \Vert \sqrt{h^{k-1}}\bar u^{k-1} \Vert_{L^{2}}^2+\varepsilon\Vert {h^{k-1}}\bar u^{k-1}_r\Vert_{ L^2}^2\\
&+\sum_{i=1}^9 W_i.
\end{split}
\end{equation}

For $W_1$, with \eqref{lin k-estimates-1} and integrating by parts, we have
\begin{equation}\label{W1}
	\begin{aligned}
		W_1=&2\delta \int(\bar u_r^k+\frac{2}{r}\bar u^k)_rh^kh^k \bar u^k\\
		=&-4\delta \int \bar u^k_r  h^k h^k_r \bar u^k-2\delta \Vert h^k \bar u^k_r \Vert_{ L^2}^2+2\delta \int \frac{2}{r}\bar u^k_r h^k h^k \bar u^k-\delta\Vert h^k \frac{2}{r} \bar u^k \Vert_{ L^2}^2\\
		\leq& C\Vert  \sqrt{h^k}\bar u^k \Vert_{ L^2}^2+ \frac{\delta}{5}\Vert h^k \bar u^k_r\Vert_{ L^2}^2+\delta \int \frac{2}{r} [(\bar u^k)^2]_r h^kh^k-2\delta \Vert h^k \bar u^k_r \Vert_{ L^2}^2-\delta\Vert h^k \frac{2}{r} \bar u^k \Vert_{ L^2}^2\\
		\leq& C\Vert  \sqrt{h^k}\bar u^k \Vert_{ L^2}^2+ \frac{\delta}{5}\Vert h^k \bar u^k_r\Vert_{ L^2}^2
		+\delta\int \frac{2}{r^2}(h^k\bar u^k)^2 +C\Vert h^k_r \Vert_{ L^\infty}\Vert \sqrt{h^k}\bar u^k \Vert_{ L^2}^2\\
		&-2\delta \Vert h^k \bar u^k_r \Vert_{ L^2}^2-\delta\Vert h^k \frac{2}{r} \bar u^k \Vert_{ L^2}^2\\
		\leq&C\Vert  \sqrt{h^k}\bar u^k \Vert_{ L^2}^2+ \frac{\delta}{5}\Vert h^k \bar u^k_r\Vert_{ L^2}^2-2\delta \Vert h^k \bar u^k_r \Vert_{ L^2}^2-\frac{\delta}{2}\Vert h^k \frac{2}{r} \bar u^k \Vert_{ L^2}^2.\\
		\end{aligned}
\end{equation} 

For $W_2$ and $W_3$,
\begin{equation}\label{W23}
\begin{split}
W_2+W_3=&2\delta \int [ (u_r^{k-1}+\frac{2}{r}u^{k-1})_r\bar h^kh^k+(u_r^{k-1}+\frac{2}{r}u^{k-1})_rh^{k-1}\bar h^k] \bar u^k\\
\leq&C\Vert rh^{k-1}(u_r^{k-1}+\frac{2}{r}u^{k-1})_r \Vert_{ L^2}\Vert \bar h^k \Vert_{ L^2}\Vert \bar u^k \Vert_{ L^\infty}\leq C\Vert \bar h^k \Vert_{ L^2}^2+\frac{\delta}{5}\Vert \bar u^k \Vert_{ H^1}^2.
\end{split}
\end{equation}

For $W_4$ and $W_{7}$, with the following fact
\begin{equation*}
	\begin{aligned}
		\Vert  h^k \frac{2}{r}\bar u^k\Vert_{ L^\infty}
		\leq& \Vert  h^k \frac{2}{r}\bar u^k\Vert_{ L^2}+\Vert  h^k_r \frac{2}{r}\bar u^k\Vert_{ L^2}+\Vert  h^k \frac{2}{r^2}\bar u^k\Vert_{ L^2}+\Vert  h^k \frac{2}{r}\bar u^k_r\Vert_{ L^2}\\
		\leq&C \Vert  h^k \frac{2}{r}\bar u^k\Vert_{ L^2}+C\Vert \sqrt{h^k}\bar u^k\Vert_{ L^2}+\Vert  h^k \bar u^k_r\Vert_{ L^2},\\
	\end{aligned}
\end{equation*} we have
\begin{equation}\label{W4}
	\begin{split}
		W_4+W_7=&2 \int [\delta(\bar v^k-\bar u^{k-1})u^k_rh^k
		+(\delta-1) (\bar v^k-\bar u^{k-1})\frac{2}{r}u^{k-1}h^k]\bar u^k\\
		\leq& C\Vert  \bar v^k \Vert_{ L^2}\Vert r u^k_r \Vert_{ L^2}\Vert  h^k \frac{2}{r}\bar u^k\Vert_{ L^\infty}+C\Vert \sqrt{h^{k-1}}\bar u^{k-1} \Vert_{ L^2}\Vert u^k _r \Vert_{ L^\infty}\Vert \sqrt{h^k}\bar u^k \Vert_{ L^2}\\
		&+C\Vert \bar v^k \Vert_{ L^2}\Vert u^{k-1}\Vert_{ L^\infty}\Vert h^k\frac{2}{r}\bar u^k \Vert_{ L^2}+C\Vert \sqrt{h^k}\bar u^k\Vert_{ L^2}\Vert u^{k-1}\Vert_{ L^\infty}\Vert h^{k-1}\frac{2}{r}\bar u^{k-1}\Vert_{ L^2}\\
		\leq&C\Vert \bar v^k \Vert_{ L^2}^2+C\Vert \sqrt{h^k}\bar u^k \Vert_{ L^2}^2+\frac{\delta}{5}\Vert h^k \bar u^k_r \Vert_{ L^2}^2+\frac{\delta}{4}\Vert h^k \frac{2}{r}\bar u^k \Vert_{ L^2}^2\\
		&+\varepsilon\Vert \sqrt{h^{k-1}}\bar u^{k-1} \Vert_{ L^2}^2+\varepsilon\Vert h^{k-1}\frac{2}{r}\bar u^{k-1}\Vert_{ L^2}^2.
	\end{split}
\end{equation}

For $W_5$ and $W_6$, we have
\begin{equation}\label{W5}
\begin{aligned}
W_5+W_6=&2\int [\delta (v^{k-1}-u^{k-2}) \bar u^k_rh^k+\delta (v^{k-1}-u^{k-2}) u^{k-1}_r\bar h^k]\bar u^k\\
\leq&C\Vert \sqrt{h^k}\bar u^k \Vert_{ L^2}^2+C\Vert \bar h^k \Vert_{ L^2}^2+\frac{\delta}{5}\Vert h^k \bar u^k_r \Vert_{ L^2}^2.
\end{aligned}
\end{equation}

For $W_8$ and $W_9$, we have
\begin{equation}\label{W8}
	\begin{aligned}
		W_8+W_9=&2\int [(\delta-1) (v^{k-1}-u^{k-2})\frac{2}{r}\bar u^{k-1}h^k+(\delta-1) (v^{k-1}-u^{k-2})\frac{2}{r}u^{k-2}\bar h^k]\bar u^k\\
		\leq&C\Vert \sqrt{h^k}\bar u^k \Vert_{ L^2}^2+C\Vert \bar h^k \Vert_{ L^2}^2+\varepsilon\Vert \sqrt{h^{k-1}}\bar u^{k-1}\Vert_{ L^2}^2.
	\end{aligned}
\end{equation}

Combining with \eqref{bar hu L2}-\eqref{W8}, we obtain at
\begin{equation}\label{final bar hu L2}
\begin{split}
&\frac{d}{dt}\Vert\sqrt{h^k} \bar u^k \Vert_{L^{2}}^2+\delta\Vert h^k\bar u^k_r\Vert_{ L^2}^2+\frac{\delta}{4}\Vert h^k \frac{2}{r}\bar u^k \Vert_{ L^2}^2\\
\leq&C\Vert \sqrt{h^k}\bar u^k \Vert_{L^{2}}^2+C\Vert \bar h^k \Vert_{ L^2}^2+C\Vert \overline{h^k\phi^k}\Vert_{ L^2}^2+C\Vert \bar v^k\Vert_{ L^2}^2\\
&+\varepsilon \Vert \sqrt{h^{k-1}}\bar u^{k-1} \Vert_{L^{2}}^2+\varepsilon\Vert {h^{k-1}}\bar u^{k-1}_r\Vert_{ L^2}^2+\varepsilon\Vert h^{k-1}\frac{2}{r}\bar u^{k-1}\Vert_{ L^2}^2.
\end{split}
\end{equation}

Finally, we complete the proof of Lemma \ref{conv of u}.
\end{proof}

\subsubsection*{Strong convergence for the full sequence}

Combining (\ref{bar ul22}) and lemmas \ref{conv of rho}-\ref{conv of u},
yields that
\begin{equation*}
\begin{split}
&\frac{d}{dt}\Gamma^k+\Vert h^k \bar u^k_r \Vert_{L^{2}}^2+\Vert h^k\frac{2}{r}\bar u^k \Vert_{L^{2}}^2\\
\leq& C_\varepsilon (\Gamma^k+\Gamma^{k-1})+\varepsilon \Vert h^{k-1} \bar u^{k-1}_r \Vert_{L^{2}}^2+\varepsilon\Vert h^{k-1}\frac{2}{r}\bar u^{k-1} \Vert_{L^{2}}^2,
\end{split}
\end{equation*}
where we define

\begin{equation*}
\begin{cases}
	\Gamma^{k}=\Vert \bar \rho^k \Vert_{L^{2}}^2+ \Vert \bar \phi^k \Vert_{L^{2}}^2+\Vert \bar v^k \Vert_{L^{2}}^2+\Vert \overline{h^k\phi^k} \Vert_{L^{2}}^2+\Vert \bar h^k \Vert_{L^{2}}^2+\Vert  \sqrt{h^k}\bar u^k \Vert_{L^{2}}^2,\\[1mm]
	\Gamma^k_*=\Vert h^k \bar u^k_r \Vert_{L^{2}}^2+\Vert h^k\frac{2}{r}\bar u^k \Vert_{L^{2}}^2.
\end{cases}
\end{equation*}

Next we choose $\varepsilon > 0$ and $T^* \in (0, {\rm min}(1, T))$ small enough, such that
$$
\varepsilon\leq \frac{1}{2C}, \ \ C_{\e} T^*\leq\frac{1}{2} ,
$$
and then, with $\Gamma^k(r,0)=\Gamma^k_*(r,0)=0$ we have
\begin{equation}\label{strong conv1}
\begin{split}
\sup \limits_{0 \leq s \leq t}\Gamma^{k}(s)+\int_{0}^{t}\Gamma_*^k(s)ds
\leq \frac{1}{2}\left( \sup \limits_{0 \leq s \leq t}\Gamma^{k-1}(s)+\int_{0}^{t}\Gamma_*^{k-1}(s)ds     \right).
\end{split}
\end{equation}

With \eqref{lin k-estimates-1}, \eqref{weak conv} and \eqref{strong conv1}, the full sequence $(\rho^k,h^k,\phi^k,v^k,u^k)$ converges strongly to the limit
$(\rho,h,\phi,v,u)$ which is stated in (\ref{weak conv}), namely,
\begin{align}\label{strong convergence}
	\begin{split}
		&\rho^k \to \rho \ {\rm in}\ L^{\infty}\big(0,T^*;L^2(r_1,r_2)\big),\\[1mm]
		&h^k \to h \ {\rm in} \ L^{\infty}\big(0,T^*;L^2(r_1,r_2)\big),\\[1mm]
		&\phi^k \to \phi\ {\rm in} \ L^{\infty}\big(0,T^*;L^2(r_1,r_2)\big),\\[1mm]
		& v^k \to v\ {\rm in} \ L^{\infty}\big(0,T^*;L^2(r_1,r_2)\big),\\[1mm]
		&u^k \to u\ {\rm in} \ L^{\infty}\big(0,T^*;L^2(r_1,r_2)\big)\cap L^2\big(0,T^*;D^1(r_1,r_2)\big),
	\end{split}
\end{align}
as $k\rightarrow\infty$ for any given $r_2>r_1>a$.

With (\ref{weak conv}) and (\ref{strong convergence}), we pass the corresponding sequences in (\ref{approximate sys}) and (\ref{approximate initial}) to the limits and obtain that $(\rho, h, \phi, v, u)$ is the desired strong solution to (\ref{reform sys})-(\ref{compatibility condition local}). The proof of Theorem \ref{localthm2} is complete.

\subsection{Recovering the original system from the auxiliary one with	$\eta>0$}\label{sec 2.5}

According to the initial conditions $\tilde \phi_0=(\tilde \rho_0)^{\gamma-\delta}$, $\tilde h_0=2(\tilde \rho_0)^{\delta-1}$ and \eqref{reform sys}$_{1,2,3}$, we can esaily obtain
\begin{equation*}
	\begin{cases}
	h=2\rho^{\delta-1},\\
	\phi=\rho^{\gamma-\delta}.
	\end{cases}
\end{equation*} 

Next, subtracting (\ref{reform sys})$_4$ from (\ref{reform sys})$_5$ with $h>0$, we obtain that
\begin{equation}\label{equ of v-u}
	(v-u)_t+u(v-u)_r+\delta h(u_r+\frac{2}{r}u)_r+\delta(v-u)u_r+(\delta-1)(v-u)\frac{2}{r}u=0.
\end{equation} 

Differentiating \eqref{reform sys}$_2$ with respect to $r$ and multiplying $\frac{\delta}{\delta-1}$, we have
\begin{equation}\label{equ of hrt}
	\frac{\delta}{\delta-1}h_{rt}+\frac{\delta}{\delta-1}h_{rr}u+\frac{\delta^2}{\delta-1}h_ru_r+\delta h(u_r+\frac{2}{r}u)_r+\delta h_r\frac{2}{r}u=0,
\end{equation} and then subtracting \eqref{equ of v-u} from \eqref{equ of hrt}, one has
\begin{equation}\label{equ of v-u-hrt}
	(v-u-\frac{\delta}{\delta-1}h_r)_t+u(v-u-\frac{\delta}{\delta-1}h_r)_r+\delta(v-u-\frac{\delta}{\delta-1}h_r)u_r+(\delta-1)(v-u-\frac{\delta}{\delta-1}h_r)\frac{2}{r}u=0.
\end{equation} 

Together with \eqref{equ of v-u-hrt} and \eqref{lin initial} $\tilde v_0-u_0-\frac{\delta}{\delta-1} \tilde h_{0,r}=0$, we have
$$
v-u=\frac{\delta}{\delta-1}h_r=2\delta \rho^{\delta-2}\rho_r.
$$

\subsection{Passing to the limit $\eta \to 0$}\label{sec 2.6}
In this subsection, we will complete the proof of Theorem \ref{localthm}. We consider the system as follows (we denote the solution ($\rho,u$) mentioned in the previous section as ($\rho^\eta,u^\eta$)):
\begin{equation}\label{eta sys}
	\left\{
	\begin{aligned}
		&\rho^\eta_t + \rho^\eta_ru^\eta+\rho ^\eta u_r^\eta+\frac{2}{r} \rho^\eta u^\eta= 0,\\
	&(\rho^\eta u^\eta)_t + \frac{1}{r^2} \big( r^2 \rho^\eta ( u^\eta)^2 \big)_r + P(\rho^\eta)_r= 2\delta (\rho^\eta)^\delta (u^\eta_r+\frac{2}{r} u^\eta)_r \\
	&+2\delta (\rho^\eta)^{\delta-1}\rho^\eta_r u^\eta_r+2\delta(\delta-1)(\rho^\eta)^{\delta-1}\rho^\eta_r (u^\eta_r+\frac{2}{r}u^\eta),\\
		&(\rho^\eta,u^\eta)(x, 0) = (\tilde \rho_0, u_0 )(x).
	\end{aligned}
	\right.
\end{equation}



Since the estimates in Proposition \ref{prop2} are independent of $\eta$, then there exists a subsequence (still denoted by $(\rho^\eta, u^\eta)$) converging to the limit $(\rho,u)$ as $\eta \to 0$ in weak$^*$ sense as follows:
\begin{equation}\label{eta conv}
\begin{split}
\rho^\eta-\eta \rightharpoonup \rho\ \ &\text{weakly*}\ \ \text{in}\ \  L^\infty(0,T^*;H^2),\\
\rho^\eta_t \rightharpoonup \rho_t\ \ &\text{weakly*}\ \ \text{in}\ \  L^\infty(0,T^*;H^1),\\
u^\eta \rightharpoonup u \ \ &\text{weakly*}\ \ \text{in}\ \  L^\infty(0,T^*;H^2)\cap L^2(0,T^*;D^3),\\
u^\eta_t \rightharpoonup u_t\ \ &\text{weakly*}\ \ \text{in}\ \  L^\infty(0,T^*;L^2)\cap L^2(0,T^*;D^1),\\
[(\rho^\eta)^{\delta-1}]_r \rightharpoonup (\rho^{\delta-1})_r\ \ &\text{weakly*}\ \ \text{in}\ \  L^\infty(0,T^*;H^1),
\end{split}
\end{equation}
where the weak limit in (\ref{eta conv})$_5$ is exactly $(\rho^{\delta-1})_r$ for $\delta-1<0$, since $\rho(r,t)$ is positive for any $(r,t)\in I_a\times[0,T^*]$ by using the continuity equation and the initial condition $\rho_0>0$.

In addition, with (\ref{eta conv}), Aubin-Lions Lemma and the lower semi-continuity, we can prove Theorem \ref{localthm}. Besides, the uniqueness and time continuity are standard, please refer for instance to \cite{Cho04}. Then the proof of Theorem \ref{localthm} is complete.

\begin{remark}\label{re1}
Similar to \cite{C-L-Z, W-Z SIAM 2025}, letting $(\rho,u)$ be the solution obtained in Theorem \ref{localthm}, we have
\begin{equation}\label{re}
 r^{2+\alpha} \rho \in L^\infty([0,T^*];L^1) \ \ {\rm{if}} \ \  r^{2+\alpha} \rho_0 \in L^1 \ \ {\rm{additionally}},
\end{equation} for any fixed $\alpha \in (1,2)$.

In fact,
Let $\zeta : \mathbb{R}^+ \to \mathbb{R}$ satisfy
\begin{equation*}
\zeta(s)=
\begin{cases}
\displaystyle1, &s\in [0,\frac{1}{2}],\\
\displaystyle(16-\frac{20}{e})s^3+(\frac{44}{e}-36)s^2+(24-\frac{29}{e})s+\frac{6}{e}-4, &s\in [\frac{1}{2},1],\\
\displaystyle e^{-s}, &s\in[1,+\infty),
\end{cases}
\end{equation*}
where $\zeta \in C^1$ similar to that in \cite{C-L-Z,W-Z SIAM 2025}. Then there exists a generic constant $C>0$ such that $|\zeta'(s)|\leq C\zeta(s)$.
Next, we define $\zeta_R(r)=\zeta(\frac{|r|}{R})$. Multiplying (\ref{system1D})$_1$ by $\zeta_R \cdot r^{2+\alpha}$, we have
\begin{equation}\label{xrho 1}
(\zeta_R r^{2+\alpha} \rho)_t+\zeta_R r^\alpha(r^2\rho u)_r=0.
\end{equation}
Integrating (\ref{xrho 1}) by parts over $I_a \times [0,t]$ for $0\leq t < T^*$, we obtain that
\begin{equation}\label{xrho 2}
\begin{split}
&\int r^{2+\alpha}\zeta_R  \rho =\int r^{2+\alpha} \zeta_R \rho_0+\int_{0}^{t}\int r^2\rho u (\zeta_R r^\alpha)_r \\
\leq& \int r^{2+\alpha} \zeta_R \rho_0+ \int_{0}^{t}\int| r^2 \rho u \zeta_R(r^\alpha)_r|+\int_{0}^{t}\int | r^{2+\alpha}\rho u (\zeta_R)_r|\\
\leq& C\int   r^{2+\alpha} \rho_0+\frac{\alpha}{a}\int_{0}^{t}\Vert u \Vert_{ L^\infty}\int r^{2+\alpha} \zeta_R \rho+\frac{C}{R} \int_{0}^{t}\Vert u \Vert_{ L^\infty}\int r^{2+\alpha}\zeta_R \rho.
\end{split}
\end{equation}
By virtue of the initial condition $ r^{2+\alpha}\rho_0 \in L^1$, (\ref{xrho 2}) and Gronwall's inequality, we obtain that
\begin{equation*}\label{xrho 3}
\int r^{2+\alpha}  \zeta_R  \rho\leq  C.
\end{equation*}
Then passing to the limit $R \to +\infty$, we get (\ref{re}).
 
\end{remark}

\section{Global-in-time existence and uniqueness}\label{sec3}
Fix $T>0$, the main purpose in this section is to establish the global-in-time {\it a priori} estimates for the
solutions which are mentioned in Theorem \ref{localthm}. Let $C$ denote a generic positive constant depending on $T$, $\rho_0$, $u_0$, $\gamma$, $\delta$, $a$ and some other known constants.

\subsection{Some basic entropy estimates}
Firstly, similar to Lemma 3.1 in \cite{W-Z SIAM 2017}, we obtain the following lemma.
\begin{lemma}\label{le:rho L1 estimate}
For any $t\in [0,T]$, by the continuity equation \eqref{system1D}$_1$, we have
\begin{equation}\label{rho L1 estimate}
\int  r^2 \rho(\cdot,t)dr=\int r^2 \rho_0 dr\leq C.
\end{equation}
\end{lemma}

Then, we establish the standard energy estimate and BD entropy estimate.
\begin{lemma}\label{le:basic estimates}
For $\gamma>1$, $\frac{2}{3} \leq \delta<1$ and any $t\in [0,T]$, there holds

\begin{equation}\label{basic estimates}
\begin{split}
&\int (\frac{r^2}{2} \rho u^2+ \frac{r^2P}{\ga-1})(\cdot,t)dr+ \int _0^t \int \rho^\delta \left[  (2\delta-\frac{4}{3})r^2(u_r+\frac{2}{r}u)^2+\frac{4}{3}r^2(u_r-\frac{u}{r})^2      \right]drds\leq C,\\
&\int (\frac{r^2}{2} \rho v^2+ \frac{r^2P}{\ga-1})(\cdot,t)dr+2\gamma \delta \int_{0}^{t}\int r^2 \rho^{\gamma+\delta-3}\rho_r^2drds\leq C,
\end{split}
\end{equation}
where $v=u+2\delta\rho^{\delta-2}\rho_r$.
\begin{proof}
	Multiplying the first two terms of \eqref{system1D}$_2$ by $r^2 u$ and using \eqref{system1D}$_1$, we obtain that
	\begin{equation*}\label{equ u r2u}
		\begin{aligned}
			&r^2(\rho u)_t u+(r^2 \rho u^2)_r u=r^2 \rho_t u^2+r^2\rho uu_t+2r \rho u^3+r^2 \rho_r u^3+2r^2 \rho u^2 u_r\\
			=&(\frac{1}{2} r^2 \rho_t u^2+r^2\rho uu_t)-\frac{1}{2} r^2 (\rho_r u+\rho  u_r +\frac{2}{r}\rho u) u^2+2r \rho u^3+r^2 \rho_r u^3+2r^2 \rho u^2 u_r\\
			=&\frac{1}{2}\frac{d}{dt}(r^2 \rho u^2)+\frac{1}{2}r^2 \rho _r u^3+ \frac{3}{2}r^2\rho u^2 u_r  +r\rho u^3=\frac{1}{2}\frac{d}{dt}(r^2 \rho u^2)+\frac{1}{2}(r^2 \rho u^3)_r.
		\end{aligned}
	\end{equation*} and then, multiplying \eqref{system1D}$_2$ by $r^2 u$ and integrating the result over $I_a$, we have
\begin{equation}\label{gamma>1 BE 1}
	\begin{aligned}
		&\frac{1}{2}\frac{d}{dt}\int r^2\rho u^2-\int\left( 2\delta \rho^\delta  (u_r+\frac{2}{r} u)_r +2\delta \rho^{\delta-1}\rho_r u_r+2\delta(\delta-1)\rho^{\delta-1}\rho_r (u_r+\frac{2}{r}u) \right)r^2 u \\
		=&-\int r^2P_ru
		=-\gamma \int r^2 \rho^{\gamma-2} \rho_r \rho u 
		=-\frac{\gamma}{\gamma-1}\int r^2 (\rho^{\gamma-1})_r \rho u\\
		=&\frac{\gamma}{\gamma-1}\int \rho^{\gamma-1}(r^2 \rho u)_r=-\frac{\gamma}{\gamma-1}\int r^2 \rho^{\gamma-1}  \rho_t
		=-\frac{1}{\gamma-1}\frac{d}{dt}\int r^2 \rho^{\gamma},
	\end{aligned}
\end{equation} where we have used \eqref{system1D}$_1$.

Next, we estimate the second term on the left-hand side of \eqref{gamma>1 BE 1} and rewrite it in terms of $(\rho^\delta, ru_r, u)$, with $\frac{2}{3}< \delta<1$ we obtain that
\begin{equation}\label{gamma>1 BE 1-2}
	\begin{aligned}
		&-\int\left( 2\delta \rho^\delta  (u_r+\frac{2}{r} u)_r +2\delta \rho^{\delta-1}\rho_r u_r+2\delta(\delta-1)\rho^{\delta-1}\rho_r (u_r+\frac{2}{r}u) \right)r^2 u \\
		=& -\int\left( 2(\rho^\delta u_r)_r +\frac{4}{r} \rho^\delta u_r -\frac{4}{r^2}\rho^\delta u+2(\delta-1)[\rho^\delta(u_r+\frac{2}{r}u)]_r \right)r^2 u \\
		=&\int\left[2(\rho^\delta u_r)(r^2u_r+2ru)-4r \rho^\delta uu_r +4\rho^\delta u^2+2(\delta-1)[\rho^\delta(u_r+\frac{2}{r}u)](r^2u_r+2ru) \right]\\
		=&\int \rho^\delta \left[ 2\delta (ru_r)^2+(8\delta-4)u^2+(8\delta-8)ruu_r       \right] \\
		=&\int \rho^\delta \left[  (2\delta-\frac{4}{3})r^2(u_r+\frac{2}{r}u)^2+\frac{4}{3}r^2(u_r-\frac{u}{r})^2      \right]\ge 0,
	\end{aligned}
\end{equation} where $(u_r+\frac{2}{r}u)$ denotes the volumetric expansion rate and $(u_r-\frac{u}{r})$ denotes the shear deformation rate. Then, integrating \eqref{gamma>1 BE 1} over $[0,t]$ implies \eqref{basic estimates}$_1$.
 
 According to \eqref{system1D}, we have that
\begin{equation}\label{equ v}
\rho(v_t+uv_r)+P_r=0,
\end{equation} and then, multplying \eqref{equ v} by $r^2v$ and integrating over $I_a$, similar to \eqref{equ u r2u} we have
\begin{equation}\label{gamma>1 BE 2}
	\begin{aligned}
		\frac{d}{dt}\int(\frac{r^2}{2}\rho v^2)=-\int r^2 P_r (u+2\delta \rho^{\delta-2}\rho_r)=-\frac{d}{dt} \left( \frac{r^2P}{\gamma-1} \right)-2 \gamma \delta\int   r^2 \rho^{\gamma+\delta-3} \rho_r^2.
	\end{aligned}
\end{equation} Finally, integrating \eqref{gamma>1 BE 1} over $[0,t]$ implies \eqref{basic estimates}$_2$.

\end{proof}

\begin{lemma}\label{le:ga=1}
For $\gamma=1$, $\frac{2}{3}<\delta<1$, $1<\alpha<2$ and any $t\in [0,T]$, there holds
\begin{equation}\label{ga=1}
\begin{cases}
\displaystyle\int(r^2\rho u^2+r^{2+\alpha}\rho)(\cdot,t)+\int_0^t \int \rho^\delta \left[  (2\delta-\frac{4}{3})r^2(u_r+\frac{2}{r}u)^2+\frac{4}{3}r^2(u_r-\frac{u}{r})^2      \right]drds \leq C,\\[2mm]
\displaystyle\int (r^2\rho v^2+r^{2+\alpha}\rho)(\cdot,t)+C\int_{0}^{t}\int r^2 \rho^{\gamma+\delta-3}\rho_r^2drds\leq C.
\end{cases}
\end{equation}
\end{lemma}
\begin{proof}
The proof of (\ref{ga=1})$_1$ is similar to that in \cite{H-Z 2024}. 
Multiplying (\ref{system1D})$_2$ by $r^2u$, integrating by parts over $I_a$, with (\ref{system1D})$_1$ we obtain that
\begin{equation}\label{ga=1 1}
\begin{split}
&\frac{1}{2}\frac{d}{dt}\int r^2\rho u^2+\int \rho^\delta \left[  (2\delta-\frac{4}{3})r^2(u_r+\frac{2}{r}u)^2+\frac{4}{3}r^2(u_r-\frac{u}{r})^2      \right]\\
=& \int \rho (r^2u_r+2r u) =\int r^2\rho (u_r+\frac{2}{r} u)=-\int r^2 \rho(\frac{\rho_t}{\rho}+u\frac{\rho_r}{\rho}) \\
 =&-\int [r^2 \rho (\ln \rho)_t-( r^2 \rho u)_r \ln \rho] =-\int[ r^2 \rho (\ln \rho)_t+ r^2 \rho_t \ln \rho] =-\frac{d}{dt}\int  r^2 \rho\ln \rho .
\end{split}
\end{equation}

Next, integrating (\ref{ga=1 1}) over $[0,t]$ for any $0\leq t \leq T$, we have
\begin{equation}\label{ga=1 2}
\begin{split}
\frac{1}{2}\int r^2\rho u^2+\int_0^t \int \rho^\delta \left[  (2\delta-\frac{4}{3})r^2(u_r+\frac{2}{r}u)^2+\frac{4}{3}r^2(u_r-\frac{u}{r})^2      \right]drds\leq C-\int r^2 \rho\ln \rho(\cdot,t) dr.
\end{split}
\end{equation}
In fact, for $\alpha>1$ we have
\begin{equation}\label{ga=1 3}
\begin{split}
&-\int r^2 \rho\ln \rho(\cdot,t) =-\int_{\rho \leq1}  r^2 \rho\ln \rho(\cdot,t) -\int_{\rho>1} r^2 \rho\ln \rho(\cdot,t)  \leq -\int_{\rho \leq1} r^2 \rho\ln \rho(\cdot,t) \\
=&\int_{\rho \leq1} r^2 \rho\ln\frac{1}{\rho}(\cdot,t) \leq C\int_{\rho \leq1} r^2 \rho^{\frac{3}{4}}(\cdot,t) \leq C\Vert r^{\frac{2+\alpha}{4}}\rho^{\frac{1}{4}}\Vert_{ L^4}\Vert   r^{1+\frac{\alpha}{2}} \rho^\frac{1}{2}\Vert_{ L^2}\Vert  r^{\frac{2-3\alpha}{4}} \Vert_{ L^4}
\leq C + C\int   r^{2+\alpha} \rho.
\end{split}
\end{equation} where we have used Lemma \ref{le:rho L1 estimate} and $\displaystyle \ln \frac{1}{\rho}\leq C(\frac{1}{\rho})^{\frac{1}{4}}$.

Putting (\ref{ga=1 3}) to (\ref{ga=1 2}), we have
\begin{equation}\label{ga=1 4}
\frac{1}{2}\int r^2\rho u^2+\int_0^t \int \rho^\delta \left[  (2\delta-\frac{4}{3})r^2(u_r+\frac{2}{r}u)^2+\frac{4}{3}r^2(u_r-\frac{u}{r})^2      \right]drds \leq C+ \int   r^{2+\alpha} \rho.
\end{equation}

For the last term of \eqref{ga=1 4}, for $\alpha<2$ and $r>a>0$, we have
\begin{equation}\label{ga=1 5}
\begin{split}
&\frac{d}{dt}\int r^{2+\alpha} \rho = \int  r^2  \rho_t  r^{\alpha}  =-\int  (r^2\rho u)_r r^{\alpha} \\
\leq&C\int \rho |u|   r^{1+\alpha}
\leq C\Vert  r^{\alpha} \rho^{\frac{1}{2}} \Vert_{ L^2}\Vert r \rho^{\frac{1}{2}}|u|\Vert_{ L^2}\\
\leq&C\int  r^{2\alpha} \rho +C\int  r^2 \rho u^2 \leq C\int  r^{2+\alpha} \rho +C\int  r^2 \rho u^2.
\end{split}
\end{equation}

Then
\begin{equation}\label{ga=1 6}
\int r^{2+\alpha}\rho \leq C\int_{0}^{t}\int  r^{2+\alpha} \rho drds+C\int_{0}^{t} \int r^2 \rho u^2 drds.
\end{equation}

Next, together with (\ref{ga=1 6}) and (\ref{ga=1 4}) and using Gronwall's inequality, we obtain that
\begin{equation}\label{new0}
\int(r^2\rho u^2+r^{2+\alpha}\rho )(\cdot,t)+\int_0^t \int \rho^\delta \left[  (2\delta-\frac{4}{3})r^2(u_r+\frac{2}{r}u)^2+\frac{4}{3}r^2(u_r-\frac{u}{r})^2      \right]drds \leq C,
\end{equation}
 and (\ref{ga=1})$_2$'s proof is similar.
\end{proof}

\end{lemma}

Next, we show the conservation of total mass and the uniform upper bound of the density.

\begin{lemma}\label{le:upper bound of rho}
For $\gamma\ge1$, $\frac{2}{3}\leq \delta<1$ and any $t\in [0,T]$, there holds
\begin{equation}\label{upper bound of rho}
\begin{split}
\Vert \rho (\cdot,t)\Vert_{L^{\infty}}\leq C.
\end{split}
\end{equation}
\end{lemma}
\begin{proof}
	According to Lemma \ref{le:rho L1 estimate}-\ref{le:ga=1} and $v-u \sim \rho^{\delta-2}\rho_r$, for $\gamma \ge 1$ we have that
	$$
	\Vert ( \rho^{\delta-\frac{1}{2}})_r \Vert_{ L^2} +\Vert \rho^{\delta-\frac{1}{2}}\Vert_{L^{\frac{2}{2\delta-1}}}\leq C,
	$$ where $\frac{2}{2\delta-1}>2$.
	
We define that $$R=\rho^{\delta-\frac{1}{2}},$$and prove this lemma in two steps: 

\textbf{Step 1: Estimate on $I_1 = [a, a+1]$.}\\
	
	  With $|I_1|=1$ and H\"older's inequality we have
	  \begin{equation*}
	  	\begin{cases}
	  		\displaystyle	\Vert R_r \Vert_{L^1(I_1)} \leq |I_1|^{1/2} \Vert R_r \Vert_{L^2(I_1)} \leq \Vert R_r \Vert_{L^2},\\[2mm]
	  			\displaystyle	\Vert R \Vert_{L^1(I_1)} \leq |I_1|^{\frac{3-2\delta}{2}} \Vert R \Vert_{L^{\frac{2}{2\delta-1}}(I_1)} \leq \Vert R \Vert_{L^{\frac{2}{2\delta-1}}}.
	  	\end{cases}
	  \end{equation*}
	
	The Sobolev embedding $W^{1,1}(I_1) \hookrightarrow L^\infty(I_1)$ yields that
\begin{equation*}
		\Vert R \Vert_{L^\infty(I_1)} \leq C ( \Vert R \Vert_{L^{\frac{2}{2\delta-1}}} + \Vert R_r \Vert_{L^2} ) \leq C .
\end{equation*}

\textbf{Step 2: Estimate on $(a+1, \infty)$.}\\
	
	Take any $r > a+1$ and consider the unit interval centered at $r$:
$$
	I_r = (r - \frac{1}{2}, r + \frac{1}{2}).
$$
	Since $r > a+1$, we have $I_r \subset [a, \infty)$.
	Let
$$
\displaystyle	\bar{R} = \frac{1}{|I_r|} \int_{I_r} R(s)ds = \int_{I_r} R(s)ds.
$$
With Poincar\'e inequality, we have
\begin{equation}\label{R Linfty 1}
	\Vert R - \bar{R} \Vert_{L^2(I_r)} \leq C \Vert R_r \Vert_{L^2(I_r)}.
\end{equation}
Then, \eqref{R Linfty 1} yields that
\begin{equation}\label{R Linfty 2}
		\Vert R - \bar{R} \Vert_{L^\infty(I_r)} \leq C ( \Vert R - \bar{R} \Vert_{L^2(I_r)} + \Vert R_r \Vert_{L^2(I_r)} )\leq C\Vert R_r \Vert_{L^2(I_r)}.
\end{equation}

By H\"older's inequality, we have
\begin{equation}\label{R Linfty 3}
	|\bar{R}| \leq \frac{1}{|I_r|} \int_{I_r} |R(s)|ds = \Vert R\Vert_{L^1(I_r)} \leq C \Vert R \Vert_{L^{\frac{2}{2\delta-1}}(I_r)}.
\end{equation}

From \eqref{R Linfty 2} and \eqref{R Linfty 3},
$$
	|R(r)| \leq |\bar{R}| + |R(r) - \bar{R}| \leq \Vert R \Vert_{L^{\frac{2}{2\delta-1}}(I_r)} + C \Vert R_r \Vert_{L^2(I_r)}\leq  \Vert R \Vert_{L^{\frac{2}{2\delta-1}}} + C \Vert R_r \Vert_{L^2}\leq C.
$$

Then, we complete the proof of this lemma.
\end{proof}

\subsection{Some new global {\it a priori} estimates}
First of all, following the idea of \cite{C-L-Z CVPDE, W-Z SIAM 2025}, we establish a auxiliary lemma before estimating $\Vert ru \Vert_{L^2}$.
	
\begin{lemma}\label{le:rho up+vp} 
For $\gamma\ge1$, $\delta \ge \frac{2}{3}$ and any $t\in [0,T]$, there holds
 \begin{equation}\label{rho up+vp}
\begin{split}
&\Vert r^{\frac{2}{p}} \rho^\frac{1}{p}u(\cdot,t)\Vert_{L^{p}}^{p} + \Vert  r^{\frac{2}{p}} \rho^\frac{1}{p}v(\cdot,t) \Vert_{L^{p}}^{p}	+C\int_{0}^{t}\Vert r\rho^\frac{\delta}{2}u_r|u|^{\frac{p}{2}-1}(\cdot,s)\Vert_{L^2}^2ds  \leq  C,
\end{split}
\end{equation}
where \begin{equation*}
	\displaystyle \frac{2\delta(2\delta-1)}{(1-\delta)^2} \triangleq K(\delta),\ \ 	2\leq p \leq \frac{K(\delta)+\sqrt{K(\delta)^2-4K(\delta)}}{2}.
\end{equation*}
\end{lemma}
\begin{proof}
Multiplying (\ref{system1D})$_2$ by $pr^2 |u|^{p-2}u$, integrating the result by parts over $I_a$, and using Young's inequality and Lemma \ref{le:upper bound of rho}, with $P_r=C\rho^{\gamma-\delta+1}(v-u)$ we have
\begin{equation}\label{rho up}
\begin{split}
\frac{d}{dt} \Vert  r^{\frac{2}{p}}\rho^\frac{1}{p}u \Vert_{L^{p}}^{p}=&-p\int r^2 |u|^{p-2}u P_r+2\delta p\int r^2 |u|^{p-2}u (\rho^\delta u_r)_r +2(\delta-1)p\int r^2 |u|^{p-2}u (\rho^\delta \frac{2}{r} u)_r\\
&+p\int r^2 |u|^{p-2}u \frac{4}{r} \rho^\delta u_r-p\int r^2 |u|^{p} \frac{4}{r^2}\rho^\delta\\
\leq& C\int | |u|^{p-2}u | |v-u| |\rho^{\gamma-\delta+1}|-2\delta p(p-1)\int r^2 \rho^\delta |u|^{p-2}u_r^2   -4\delta p \int r \rho^\delta |u|^{p-2}u  u_r\\
&-8(\delta-1)p \int\rho^\delta |u|^p -4(\delta-1)p(p-1)\int  r\rho^\delta  |u|^{p-2}u_r u\\
&+4p\int r\rho^\delta |u|^{p-2}u  u_r-4p\int \rho^\delta |u|^{p}\\
\leq& C\int | |u|^{p-2}u | |v-u| |\rho^{\gamma-\delta+1}|-2\delta p(p-1)\int r^2 \rho^\delta |u|^{p-2}u_r^2   -4(\delta-1) p^2 \int r \rho^\delta |u|^{p-2}u  u_r\\
&-(8\delta-4)p \int\rho^\delta |u|^p \\
\leq&-C\int r^2 \rho^\delta |u|^{p-2}u_r^2- C\int\rho^\delta |u|^p+C(1+\Vert \rho \Vert_{ L^\infty}^{\gamma-\delta})(\Vert  r^{\frac{2}{p}}\rho^\frac{1}{p}u \Vert_{L^{p}}^{p}+\Vert  r^{\frac{2}{p}}\rho^\frac{1}{p}v \Vert_{L^{p}}^{p}).
\end{split}
\end{equation}where we need 
\begin{equation*}
	\begin{aligned}
	[4(\delta-1)p^2]^2\leq& 4[2\delta p(p-1)]\cdot[(8\delta-4)p],\\
	\end{aligned}
\end{equation*}i.e.
$$
\frac{p^2}{p-1}\leq \frac{2\delta(2\delta-1)}{(1-\delta)^2}\triangleq K(\delta) \Rightarrow p^2-K(\delta)p+K(\delta)\leq 0.
$$

Next, we need $K^2-4K\ge0$, i.e.
\begin{equation*}
	\begin{aligned}
		K(\delta)=\frac{\delta(8\delta-4)}{2(1-\delta)^2}\ge 4 \Rightarrow \delta \ge \frac{2}{3} \Rightarrow 
			\begin{cases}
			\displaystyle	2\leq  p ,\\
			\displaystyle	\frac{K-\sqrt{K^2-4K}}{2} \leq p \leq \frac{K+\sqrt{K^2-4K}}{2},
			\end{cases}
	\end{aligned}
\end{equation*} and 
\begin{equation}
	\begin{cases}
	\displaystyle	2\leq \frac{K+\sqrt{K^2-4K}}{2} \Rightarrow K\ge 4 \Longleftrightarrow \delta \ge \frac{2}{3},\\
	\displaystyle	\frac{K-\sqrt{K^2-4K}}{2} \leq 2 \Rightarrow K\ge 4 \Longleftrightarrow \delta \ge \frac{2}{3},
	\end{cases}
\end{equation} which yields that
\begin{equation}\label{range p}
	2 \leq p \leq  \frac{K+\sqrt{K^2-4K}}{2}.
\end{equation}

To handle the last term of (\ref{rho up}), we have to estimate $\Vert  r^{\frac{2}{p}}\rho^\frac{1}{p}v \Vert_{L^{p}}^{p} $. According to (\ref{system1D}), we have that
\begin{equation}\label{v}
\rho v_t+\rho u v_r +C\rho^{\gamma-\delta+1}(v-u)=0.
\end{equation}

Multiplying $(\ref{v}$) by $p r^2 |v|^{p-2}v$ ($p \ge 2$), integrating the result over $I_a$, and using Young's inequality and Lemma \ref{le:upper bound of rho}, we obtain
\begin{equation}\label{rho vp}
\begin{split}
\frac{d}{dt} \Vert r^{\frac{2}{p}} \rho^\frac{1}{p}v \Vert_{L^{p}}^{p}\leq C\Vert \rho^{\gamma-\delta}\Vert_{L^{\infty}}(\Vert r^{\frac{2}{p}} \rho^{\frac{1}{p}}u\Vert_{L^{p}}^{p} +  \Vert r^{\frac{2}{p}} \rho^{\frac{1}{p}}v\Vert_{L^{p}}^{p} ).
\end{split}
\end{equation}

Then following from the initial conditions:
\begin{equation*}
\begin{split}
&\Vert r^{\frac{2}{p}} (u_0,v_0)\Vert_{L^{p}}\leq C\Vert r^{\frac{2}{p}} (u_0,v_0)\Vert_{H^1}\leq C\Vert r (u_0,v_0)\Vert_{H^1}\leq C,\\
&\Vert  r^2\rho_0 (u_0)^{p}\Vert_{L^1}\leq \Vert  \rho_0\Vert_{L^\infty}\Vert  r^{\frac{2}{p}}u_0 \Vert_{L^{p}}^{p}\leq C,\\
&\Vert  r^2 \rho_0 (v_0)^{p}\Vert_{L^1}\leq \Vert  \rho_0\Vert_{L^\infty}\Vert  r^{\frac{2}{p}}v_0 \Vert_{L^{p}}^{p}\leq C,
\end{split}
\end{equation*} 
together with (\ref{rho up}), (\ref{rho vp}) and Gronwall's inequality, we can prove Lemma \ref{le:rho up+vp}.
\end{proof}



With the above auxiliary lemma, we're ready to estimate $\Vert  ru \Vert_{L^2}+\Vert rv \Vert_{ L^2}$.
\begin{lemma}\label{le:uLP}
For $\gamma\ge1$, any $t\in [0,T]$ and \begin{equation}\label{delta}
\frac{2}{3}\leq \delta \leq 1,\ \ 	\frac{4-2\delta}{1-\delta}\leq \frac{K+\sqrt{K^2-4K}}{2},
\end{equation} we have
\begin{equation}\label{uLp}
\Vert ru(\cdot,t) \Vert_{L^2}^2 + \Vert rv(\cdot,t) \Vert_{L^2}^2+\int_0^t \Vert \rho^{\frac{\delta-1}{2}}u(\cdot,s)\Vert_{L^2}^2(s)ds+\int_0^t \Vert r\rho^{\frac{\delta-1}{2}}u_r(\cdot,s)\Vert_{L^2}^2(s)ds\leq C.
\end{equation}

\end{lemma}	
\begin{proof}
 Firstly, according to (\ref{system1D})$_2$ and the fact that $v= u + 2\delta\rho^{\delta-2}\rho_r$, we have
 \begin{equation}\label{u-2}
 u_t + uu_r + \frac{\gamma}{2\delta} \rho^{\gamma-\delta}(v-u) = 2\delta \rho^{\delta-1}  (u_r+\frac{2}{r} u)_r +(v-u) u_r+(\delta-1)(v-u)(u_r+\frac{2}{r}u).
 \end{equation}

Multiplying (\ref{u-2}) by $2r^2u$, integrating the result by parts over $I_a$, we have
\begin{equation}\label{Lp u}
\begin{split} 		
\frac{d}{dt}\Vert ru \Vert_{L^2}^2 
\leq& 2\int |r^2u^2u_r|+C\int |r^2\rho^{\gamma-\delta}(v-u) u|+4\delta \int r^2  \rho^{\delta-1}  (u_r+\frac{2}{r} u)_r u \\
&+2\int r^2(v-u) u_ru+2(\delta-1)\int r^2(v-u)(u_r+\frac{2}{r}u)u \\
=&G_1+G_2+G_3+G_4+G_5.\\
\end{split} 		
\end{equation}	

For $G_1$, with Young's inequality, Lemma \ref{le:rho up+vp} and choosing
\begin{equation}\label{pp2-1}
	2\leq p=\frac{4-2\delta}{1-\delta}\leq \frac{K+\sqrt{K^2-4K}}{2},
\end{equation} we have
\begin{equation}\label{G1}
	\begin{aligned}
	G_1=& C\int |r^2u^2u_r|\leq \varepsilon \int r^2  \rho^{\delta-1}  u_r^2+C\Vert r^2 \rho^{1-\delta} u^4\Vert_{ L^1}\\
	\leq& \varepsilon \int r^2  \rho^{\delta-1}  u_r^2+C\Vert r^{\frac{2}{p}}\rho^{\frac{1}{p}}u\Vert_{ L^{p}}^{p(1-\delta)}\Vert r u\Vert_{ L^2}^{4-p(1-\delta)}\\
	\leq& C(1+\Vert r u \Vert_{ L^2}^2)+\varepsilon \int r^2  \rho^{\delta-1}   u_r^2,
	\end{aligned}
\end{equation} and \eqref{pp2-1} yields that
$$
p(1-\delta)=4-2\delta >2.
$$

For $G_2$, with Lemma \ref{le:upper bound of rho} we have
\begin{equation}\label{G2}
	\begin{aligned}
		G_2=&C\int |r^2\rho^{\gamma-\delta}(v-u) |u|
		\leq C\Vert ru \Vert_{ L^2}^2+C\Vert rv \Vert_{ L^2}^2.
	\end{aligned}
\end{equation}

For $G_3$, with integrating by parts, $G_1$ and Lemma \ref{le:rho up+vp}, we choose the same $p$ and obtain that
\begin{equation}\label{G3}
	\begin{aligned}
		G_3=&4\delta \int r^2  \rho^{\delta-1}  u_{rr}u -8\delta \int   \rho^{\delta-1} u^2+8\delta \int r \rho^{\delta-1} u_r u\\
		=& -8\delta \int r  \rho^{\delta-1}  u_{r}u-4\delta \int r^2  (\rho^{\delta-1})_r  u_{r}u-4\delta \int r^2  \rho^{\delta-1}  u_{r}^2-8\delta \int   \rho^{\delta-1} u^2+8\delta \int r \rho^{\delta-1} u_r u   \\
		\leq& C \int r^2 | (v-u) u_{r}u|-4\delta \int r^2  \rho^{\delta-1}  u_{r}^2-8\delta \int   \rho^{\delta-1} u^2   \\
		\leq& \varepsilon \int r^2  \rho^{\delta-1}  u_{r}^2+ C\Vert r^2 \rho^{1-\delta} v^2 u^2\Vert_{ L^1}+C\Vert r^2 \rho^{1-\delta}  u^4\Vert_{ L^1}  -4\delta \int r^2  \rho^{\delta-1}  u_{r}^2-8\delta \int   \rho^{\delta-1} u^2   \\
		\leq& \varepsilon \int r^2  \rho^{\delta-1}  u_{r}^2+C\Vert r^{\frac{2}{p}}\rho^{\frac{1}{p}}v\Vert_{ L^{p}}^2\Vert r^{\frac{2}{p}}\rho^{\frac{1}{p}}u\Vert_{ L^{p}}^{p(1-\delta)-2}\Vert r u\Vert_{ L^2}^{4-p(1-\delta)}+C\Vert r^{\frac{2}{p}}\rho^{\frac{1}{p}}u\Vert_{ L^{p}}^{p(1-\delta)}\Vert r u\Vert_{ L^2}^{4-p(1-\delta)}\\
		&-4\delta \int r^2  \rho^{\delta-1}  u_{r}^2-8\delta \int   \rho^{\delta-1} u^2   \\
		\leq& \varepsilon \int r^2  \rho^{\delta-1}  u_{r}^2+C(1+\Vert ru \Vert_{ L^2}^2)-4\delta \int r^2  \rho^{\delta-1}  u_{r}^2-8\delta \int   \rho^{\delta-1} u^2.   \\
	\end{aligned}
\end{equation} 

For $G_4$ and $G_5$, similar to $G_3$, we have
\begin{equation}\label{G4}
	\begin{aligned}
		G_4=&2\int r^2(v-u) u_ru 
		\leq C(1+\Vert ru \Vert_{ L^2}^2)+\varepsilon \int r^2  \rho^{\delta-1}  u_r^2,\\[2mm]
		G_5=&2(\delta-1)\int r^2(v-u)(u_r+\frac{2}{r}u)u
		\leq C|G_4|+C\Vert r (v-u)u^2\Vert_{ L^1}\\
		\leq&C|G_4|+\varepsilon\int \rho^{\delta-1}u^2+C\Vert r^2 \rho^{1-\delta}v^2u^2\Vert_{ L^1}+C\Vert r^2 \rho^{1-\delta}u^4\Vert_{ L^1}\\
		\leq&\varepsilon \int   \rho^{\delta-1}  u^2+C(1+\Vert ru \Vert_{ L^2}^2). \\
	\end{aligned}
\end{equation}

Next, together with \eqref{Lp u}-\eqref{G4}, we have
\begin{equation}\label{uLP finnal}
	\begin{aligned}
		\frac{d}{dt}\Vert ru \Vert_{L^2}^2+4\delta \int   \rho^{\delta-1}  u^2+2\delta \int r^2  \rho^{\delta-1}   u_r^2
		\leq C(1+\Vert ru \Vert_{ L^2}^2)+C\Vert rv \Vert_{ L^2}^2.
	\end{aligned}
\end{equation}

By virtue of \eqref{v}, we have
\begin{equation}\label{vt}
	v_t+u v_r +C\rho^{\gamma-\delta}(v-u)=0,
\end{equation}

Multiplying (\ref{vt}) by $2r^2v$, integrating the result by parts over $I_a$, we have
\begin{equation}\label{vL2-1}
	\begin{aligned}
		\frac{d}{dt}\Vert rv \Vert_{ L^2}^2\leq& -\int r^2 u (v^2)_r +C\Vert r^2\rho^{\gamma-\delta}(v-u)v\Vert_{ L^1}\\
		\leq&C\Vert r uv^2\Vert_{ L^1}+C\Vert r^2 u_r v^2 \Vert_{ L^1}+C\Vert r^2\rho^{\gamma-\delta}(v-u)v\Vert_{ L^1}\\
		\leq&\varepsilon\int \rho^{\delta-1}u^2+ \varepsilon\int r^2\rho^{\delta-1}u_r^2 +C\Vert r^2\rho^{1-\delta} v^4 \Vert_{ L^1}+C(\Vert ru \Vert_{ L^2}^2+\Vert rv \Vert_{ L^2}^2)  \\
		\leq&C\Vert r^{\frac{2}{p}}\rho^{\frac{1}{p}}v\Vert_{ L^{p}}^{p(1-\delta)}\Vert r v\Vert_{ L^2}^{4-p(1-\delta)}+C(\Vert ru \Vert_{ L^2}^2+\Vert rv \Vert_{ L^2}^2) \leq C(1+\Vert ru \Vert_{ L^2}^2+\Vert rv \Vert_{ L^2}^2).
	\end{aligned}
\end{equation}

Together with \eqref{uLP finnal}, \eqref{vL2-1} and Gronwall's inequality we complete the proof of Lemma \ref{le:uLP}.
\end{proof}

\begin{lemma}\label{le:uLinfty}
For $\gamma\ge1$, $\delta$ satisfying \eqref{delta} and any $t\in [0,T]$, there holds
\begin{equation}\label{uLinfty}
\Vert v(\cdot,t) \Vert_{ L^\infty}+\int_{0}^{t}(\Vert u(\cdot,s) \Vert_{L^\infty}^2+\Vert \rho^{\delta-2}\rho_r (\cdot,s)\Vert_{L^\infty}^2)ds\leq C.
\end{equation}
\end{lemma}
\begin{proof}
(\ref{uLp}) together with Lemmas \ref{le:upper bound of rho}, \ref{le:uLP}, Young's inequality and the Sobolev embedding theorem $W^{1,2} \hookrightarrow L^\infty $ yields that
\begin{equation}\label{u2pLinfty}
\begin{split}
&\int_{0}^{t}\Vert u \Vert_{L^{\infty}}^2(s)ds\leq C\int_{0}^{t}\Vert u\Vert_{L^2}^2(s)ds+C\int_{0}^{t}\Vert u_r\Vert_{L^2}^2(s)ds\\
=&C\int_{0}^{t}\Vert \frac{1}{r}ru\Vert_{L^2}^2(s)+C\int_0^t\Vert \frac{1}{r}r \rho^{\frac{\delta-1}{2}}\rho^{\frac{1-\delta}{2}} u_r\Vert_{L^2}^2(s)ds\\
\leq&C\int_{0}^{t}\Vert ru \Vert_{ L^2}^2(s)ds+C\int_0^t \Vert r\rho^{\frac{\delta-1}{2}}u_r\Vert_{L^2}^2(s)ds\leq C.
\end{split}
\end{equation}

Moreover, with the standard method of characteristics, (\ref{u2pLinfty}) and Lemma \ref{le:upper bound of rho}, one obtains
\begin{equation*}
	\begin{aligned}
		\Vert v \Vert_{ L^\infty} \leq C\left(\Vert v_0 \Vert_{ L^\infty}+ \int_0^t \Vert \rho^{\gamma-\delta} u \Vert_{ L^\infty}(s)ds \right)\leq C,\\[2mm]
		\int_{0}^{t}\Vert \rho^{\delta-2}\rho_r \Vert_{L^{\infty}}^2(s)ds=C\int_{0}^{t}\Vert v-u\Vert_{L^{\infty}}^2(s)ds\leq C.
	\end{aligned}
\end{equation*}

The proof of Lemma \ref{le:uLinfty} is complete.
\end{proof}

Next, we establish the estimate of $u_r$.
\begin{lemma}\label{le:urL2}
	For $\gamma\ge1$, $\delta$ satisfying \eqref{delta} and any $t\in [0,T]$, there holds
	\begin{equation}\label{urL2}
		\Vert u(\cdot,t) \Vert_{ L^\infty}^2+\Vert ru_r(\cdot,t) \Vert_{L^2}^2+\int_0^t\int r^2\rho^{1-\delta} u_t^2 \leq C.
	\end{equation} 
\end{lemma}	
\begin{proof}
	Firstly, multiplying (\ref{u-2}) by $ r^2\rho^{1-\delta}u_t$, integrating the result by parts over $I_a$, we have
	\begin{equation}\label{urL2-1}
		\begin{aligned}
			&\int r^2\rho^{1-\delta} u_t^2-2\delta\int r^2 (u_{rr}+\frac{2}{r}u_r-\frac{2}{r^2}u)u_t\\
			=&-(1+\delta)\int r^2 \rho^{1-\delta}  uu_ru_t-C\int r^2\rho^{\gamma+1-2\delta}(v-u)u_t+\delta\int r^2\rho^{1-\delta} vu_ru_t+(\delta-1)\int r^2\rho^{1-\delta} (v-u)\frac{2}{r}uu_t\\	
			=&N_1+N_2+N_3+N_4.		 
		\end{aligned}
	\end{equation}

For the 2nd term of left-hand of \eqref{urL2-1}, with integrating by parts, we have
\begin{equation}\label{urL2-2}
	\begin{aligned}
		&-2\delta\int r^2 (u_{rr}+\frac{2}{r}u_r-\frac{2}{r^2}u)u_t\\
		=&2\delta\int 2r u_ru_t+\delta \frac{d}{dt}\Vert ru_r \Vert_{ L^2}^2-2\delta\int 2r u_ru_t+4\delta \int uu_t\\
		=&\delta \frac{d}{dt}\Vert ru_r \Vert_{ L^2}^2+2\delta\frac{d}{dt}\Vert u \Vert_{ L^2}^2.
	\end{aligned}
\end{equation}

For $N_1$, with Lemma \ref{le:upper bound of rho}, we have
\begin{equation}\label{N1}
	\begin{aligned}
		N_1\leq&\varepsilon\int r^2\rho^{1-\delta}u_t^2+C\int | r^2\rho^{1-\delta}u^2 u_r^2 |\leq \varepsilon\int r^2\rho^{1-\delta}u_t^2+C\Vert u \Vert_{ L^\infty}^2\Vert ru_r \Vert_{ L^2}^2.
	\end{aligned}
\end{equation}

For $N_2$, with Lemma \ref{le:upper bound of rho} and  \ref{le:uLP}, we have
\begin{equation}\label{N2}
	\begin{aligned}
		N_2\leq&\varepsilon\int r^2\rho^{1-\delta}u_t^2+C\int| r^2\rho^{2\gamma+1-3\delta}(v-u)^2|\leq \varepsilon\int r^2\rho^{1-\delta}u_t^2+C.
	\end{aligned}
\end{equation}

For $N_3$, with Lemma \ref{le:upper bound of rho} and \ref{le:uLinfty}, we have
\begin{equation}\label{N3}
	\begin{aligned}
		N_3\leq&\varepsilon\int r^2\rho^{1-\delta}u_t^2+C\int | r^2\rho^{1-\delta}v^2 u_r^2 |\leq \varepsilon\int r^2\rho^{1-\delta}u_t^2+C\Vert ru_r \Vert_{ L^2}^2.
	\end{aligned}
\end{equation}

For $N_4$, with Lemma \ref{le:upper bound of rho} and \ref{le:uLP} we have
\begin{equation}\label{N4}
	\begin{aligned}
		N_4\leq&\varepsilon\int r^2\rho^{1-\delta}u_t^2+C\int | \rho^{1-\delta}(v-u)^2 u^2   |\leq \varepsilon\int r^2\rho^{1-\delta}u_t^2+C(1+\Vert u\Vert_{ L^\infty}^2).
	\end{aligned}
\end{equation}

Together with \eqref{urL2-2}-\eqref{N4}, Lemma \ref{le:uLP}, \ref{le:uLinfty} and Gronwall's inequality, we have
$$
\Vert ru_r \Vert_{ L^2}^2+\int_0^t\int r^2\rho^{1-\delta} u_t^2\leq C, \ \ \text{and}\ \ \Vert u \Vert_{ L^\infty}\leq C,
$$
then, we have completed the proof of this lemma.
\end{proof}

\begin{corollary}\label{coro:hlinfty}
		For $\gamma\ge1$, $\delta$ satisfying \eqref{delta} and any $t\in [0,T]$, there holds
	\begin{equation}\label{h Linfty}
		\Vert \rho^{\delta-2}\rho_r (\cdot,t) \Vert_{L^\infty}\leq C.
	\end{equation} 
\end{corollary}

Finally, we establish the estimate of $u_t$.
\begin{lemma}\label{le:utL2}
	For $\gamma\ge1$, $\delta$ satisfying \eqref{delta} and any $t\in [0,T]$, there holds
	\begin{equation}\label{utL2}
		\Vert ru_t(\cdot,t) \Vert_{L^2}^2+\Vert rv_r(\cdot,t) \Vert_{L^2}^2+\int_0^t \Vert ru_{rrr}(\cdot,s)\Vert_{ L^2}^2ds+\int_0^t \Vert ru_{rt}(\cdot,s)\Vert_{ L^2}^2ds \leq C.
	\end{equation} 
\end{lemma}	
\begin{proof}
	Firstly, by virtue of (\ref{u-2}), Corollary \ref{coro:hlinfty}, Lemma \ref{le:upper bound of rho}, \ref{le:uLP} and \ref{le:urL2}, we have the following fact that
	\begin{equation}\label{good term}
		\begin{cases}
\displaystyle			\Vert r \rho^{\delta-1}(u_r+\frac{2}{r}u)_r\Vert_{ L^2}\leq C(1+\Vert ru_t\Vert_{ L^2}),\\
\displaystyle			\Vert r u_{rr} \Vert_{ L^2}+\Vert u_r \Vert_{ L^\infty}\leq C(1+\Vert ru_t\Vert_{ L^2}),\\
\displaystyle			\int_0^t\Vert r\rho^{\frac{1-\delta}{2}}u_{rr}\Vert_{ L^2}^2(s)ds\leq C(1+\int_0^t\Vert r\rho^{\frac{1-\delta}{2}}u_t \Vert_{ L^2}^2(s)ds)\leq C,
		\end{cases}
	\end{equation} with \eqref{good term}$_3$, we also have 
	\begin{equation}\label{good term-2}
	\begin{aligned}
		&\int_0^t\Vert r\rho^{\frac{1-\delta}{2}}u_{r}\Vert_{ L^\infty}^2(s)ds\\
		\leq& \int_0^t\Vert r\rho^{\frac{1-\delta}{2}}u_{r}\Vert_{ L^2}^2(s)ds+\int_0^t\Vert \rho^{\frac{1-\delta}{2}}u_{r}\Vert_{ L^2}^2(s)ds+\int_0^t\Vert r\rho^{\delta-2}\rho_r\rho^{\frac{3}{2}(1-\delta)} u_{r}\Vert_{ L^2}^2(s)ds+\int_0^t\Vert r\rho^{\frac{1-\delta}{2}}u_{rr}\Vert_{ L^2}^2(s)ds\\
		\leq& C(1+\int_0^t\Vert r\rho^{\frac{1-\delta}{2}}u_{rr}\Vert_{ L^2}^2(s)ds)\leq C.
	\end{aligned}
\end{equation}

	Secondly, differentiating (\ref{u-2}) with respect to $t$, we have
	\begin{equation}\label{eq utt}
		\begin{aligned}
			&u_{tt}+\frac{1}{2}(u^2)_{rt}+C(\rho^{\gamma-\delta})_t(v-u)+C\rho^{\gamma-\delta}(v_t-u_t)\\
			=&2\delta \rho^{\delta-1}(u_{rrt}+\frac{2}{r}u_{rt}-\frac{2}{r^2}u_t)+2\delta (\rho^{\delta-1})_t(u_{rr}+\frac{2}{r}u_{r}-\frac{2}{r^2}u)\\
			&+\delta (v-u)_t u_r+\delta (v-u) u_{rt}+(\delta-1)(v_t-u_t)\frac{2}{r}u+(\delta-1)(v-u)\frac{2}{r}u_t
		\end{aligned}
	\end{equation}
	
	and then, multiplying the \eqref{eq utt} by $2 r^2 u_t$ and integrating by parts over $I_a$, one arrives at
	\begin{equation}\label{utL2-1}
		\begin{aligned}
			&\frac{d}{dt}\Vert ru_t \Vert_{ L^2}^2=-C\int r^2(u^2)_{rt}u_t-C\int r^2 (\rho^{\gamma-\delta})_t(v-u)u_t-C\int r^2\rho^{\gamma-\delta}(v_t-u_t)u_t\\
			&+4\delta\int r^2 \rho^{\delta-1}(u_{rrt}+\frac{2}{r}u_{rt}-\frac{2}{r^2}u_t)u_t+4\delta\int r^2 (\rho^{\delta-1})_t(u_{rr}+\frac{2}{r}u_{r}-\frac{2}{r^2}u)u_t\\
			&+2\delta\int r^2(v-u)_t u_ru_t+C\int r^2 (v-u) u_{rt}u_t+2(\delta-1)\int r^2(v_t-u_t)\frac{2}{r}uu_t+2(\delta-1)\int r^2(v-u)\frac{2}{r}u_t^2\\&=\sum_{i=1}^9 L_i.
		\end{aligned}
	\end{equation}

For $L_1$, with integrating by parts, Lemma \ref{le:upper bound of rho}, \ref{le:urL2} and Corollary \ref{coro:hlinfty}, we have
\begin{equation}\label{L1}
	\begin{aligned}
			L_1=&-C\int r^2(u^2)_{rt}u_t
			=C\int r uu_t^2+C\int r^2 uu_t u_{rt}\\
			\leq&C\Vert u \Vert_{ L^\infty}\Vert ru_t \Vert_{ L^2}^2+\varepsilon\int r^2 \rho^{\delta-1}u_{rt}^2+C\int r^2\rho^{1-\delta}u^2 u_t^2\\
			\leq&C\Vert ru_t \Vert_{ L^2}^2+\varepsilon\int r^2 \rho^{\delta-1}u_{rt}^2.
	\end{aligned}
\end{equation}

For $L_2$, with Lemma \ref{le:upper bound of rho}, \ref{le:uLP}, \ref{le:urL2}, Corollary \ref{coro:hlinfty} and the fact that $ \gamma+1-2\delta\ge 0$, we have
\begin{equation}\label{L2}
	\begin{aligned}
		L_2=&-C\int r^2 (\rho^{\gamma-\delta})_t(v-u)u_t
		\leq C\Vert r^2 \rho^{\gamma-\delta-1}|\rho u_r+\rho_ru +\frac{2}{r}\rho u||v-u| |u_t|\Vert_{ L^1}\\
		\leq&C\Vert \rho^{\gamma-\delta}\Vert_{ L^\infty} \Vert ru_r \Vert_{ L^2}\Vert v-u\Vert_{ L^\infty} \Vert ru_t \Vert_{ L^2}+C\Vert \rho^{\delta-2}\rho_r \Vert_{ L^\infty}\Vert \rho^{\gamma+1-2\delta}\Vert_{ L^\infty} \Vert ru\Vert_{ L^2}\Vert v-u\Vert_{ L^\infty} \Vert ru_t \Vert_{ L^2}\\
		&+C\Vert \rho^{\gamma-\delta}\Vert_{ L^\infty}\Vert ru \Vert_{ L^2}\Vert v-u\Vert_{ L^\infty} \Vert ru_t \Vert_{ L^2}\\
		\leq&C(1+\Vert ru_t \Vert_{ L^2}^2).
	\end{aligned}
\end{equation}

For $L_3$, with Lemma \ref{le:upper bound of rho} and the fact that $ \gamma+1-2\delta\ge 0$, we have
\begin{equation}\label{L3}
	\begin{aligned}
		L_3=&-C\int r^2\rho^{\gamma-\delta}(v_t-u_t)u_t
		=-C\int r^2 \rho^{\gamma-\delta}(\rho^{\delta-1})_{rt}u_t\\
		=&C\int r \rho^{\gamma-\delta}(\rho^{\delta-1})_{t}u_t
		+C\int r^2 (\rho^{\gamma-\delta})_r(\rho^{\delta-1})_{t}u_t
		+C\int r^2 \rho^{\gamma-\delta}(\rho^{\delta-1})_{t}u_{rt}\\
		=&C\int r \rho^{\gamma-\delta}[(\rho^{\delta-1})_r u+ (\delta-1)\rho^{\delta-1}u_r+(\delta-1)\frac{2}{r}\rho^{\delta-1}u ]u_t\\
		+&C\int r^2 \rho^{\gamma+1-2\delta}\rho^{\delta-2}\rho_r[(\rho^{\delta-1})_r u+ (\delta-1)\rho^{\delta-1}u_r+(\delta-1)\frac{2}{r}\rho^{\delta-1}u ]u_t\\
		+&C\int r^2 \rho^{\gamma-\delta}[(\rho^{\delta-1})_r u+ (\delta-1)\rho^{\delta-1}u_r+(\delta-1)\frac{2}{r}\rho^{\delta-1}u ]u_{rt}\\
		\leq&C(\Vert ru \Vert_{ L^2}+\Vert ru_r \Vert_{ L^2})\Vert ru_t \Vert_{ L^2}+\varepsilon\int r^2 \rho^{\delta-1}u_{rt}^2+C(\Vert ru \Vert_{ L^2}^2+\Vert ru_r \Vert_{ L^2}^2)\\
		\leq&C(1+\Vert ru_t \Vert_{ L^2}^2)+\varepsilon\int r^2 \rho^{\delta-1}u_{rt}^2.
	\end{aligned}
\end{equation}

For $L_4$, with integrating by parts Lemma \ref{le:upper bound of rho} and Corollary \ref{coro:hlinfty}, we have
\begin{equation}\label{L4}
	\begin{aligned}
		L_4=&4\delta\int r^2 \rho^{\delta-1}(u_{rrt}+\frac{2}{r}u_{rt}-\frac{2}{r^2}u_t)u_t\\
		=&-4\delta\int 2r \rho^{\delta-1}u_{rt}u_t-4\delta\int r^2 (\rho^{\delta-1})_ru_{rt}u_t-4\delta\int r^2 \rho^{\delta-1}u_{rt}^2+4\delta\int 2r \rho^{\delta-1}u_{rt}u_t-8\delta\int \rho^{\delta-1}u_t^2\\
		\leq&\varepsilon\int r^2 \rho^{\delta-1}u_{rt}^2+C\int r^2 (\rho^{\delta-1})_r^2\rho^{1-\delta}u_t^2-4\delta\int r^2 \rho^{\delta-1}u_{rt}^2-8\delta\int \rho^{\delta-1}u_t^2\\
		\leq&\varepsilon\int r^2 \rho^{\delta-1}u_{rt}^2+C\Vert r u_t \Vert_{ L^2}^2-4\delta\int r^2 \rho^{\delta-1}u_{rt}^2-8\delta\int \rho^{\delta-1}u_t^2.
	\end{aligned}
\end{equation}

For $L_5$, with \eqref{good term}, \eqref{good term-2} and Corollary \ref{coro:hlinfty}, we have
\begin{equation}\label{L5}
	\begin{aligned}
		L_5=&4\delta\int r^2 (\rho^{\delta-1})_t(u_{rr}+\frac{2}{r}u_{r}-\frac{2}{r^2}u)u_t\\
		=&-C\int r^2[(\rho^{\delta-1})_r u+ (\delta-1)\rho^{\delta-1}u_r+(\delta-1)\frac{2}{r}\rho^{\delta-1}u ]  (u_{rr}+\frac{2}{r}u_{r}-\frac{2}{r^2}u)u_t\\
		\leq&C\Vert \rho^{\delta-2}\rho_r \Vert_{ L^\infty}\Vert u \Vert_{ L^\infty}(\Vert ru_{rr}\Vert_{ L^2}+\Vert ru_r\Vert_{ L^2}+\Vert ru\Vert_{ L^2})\Vert ru_t \Vert_{ L^2}\\
		&+\varepsilon\int \rho^{\delta-1}u_t^2+C\int  r^4\rho^{\delta-1}u_r^2 (u_{rr}+\frac{2}{r}u_{r}-\frac{2}{r^2}u)^2 \\
	  &+C\Vert u \Vert_{ L^\infty}\Vert r \rho^{\delta-1}(u_r+\frac{2}{r}u)_r \Vert_{ L^2} \Vert r u_t \Vert_{ L^2}\\
	  \leq&C(1+\Vert ru_{t}\Vert_{ L^2}^2)+\varepsilon\int \rho^{\delta-1}u_t^2+C\Vert r \rho^{\frac{1-\delta}{2}}u_r\Vert_{ L^\infty}^2\Vert r \rho^{\delta-1}(u_r+\frac{2}{r}u)_r \Vert_{ L^2}^2\\
	  &+C\Vert u \Vert_{ L^\infty}\Vert r \rho^{\delta-1}(u_r+\frac{2}{r}u)_r \Vert_{ L^2} \Vert r u_t \Vert_{ L^2}\\
	  \leq&C(1+\Vert r \rho^{\frac{1-\delta}{2}}u_r\Vert_{ L^\infty}^2)\Vert ru_{t}\Vert_{ L^2}^2+\varepsilon\int \rho^{\delta-1}u_t^2+C.
	\end{aligned}
\end{equation}

For $L_6$, with $L_5$, Lemma \ref{le:upper bound of rho}, \ref{le:uLP}, Corollary \ref{coro:hlinfty} and the fact that $$\displaystyle(\rho^{\delta-1})_t=-[(\rho^{\delta-1})_r u+ (\delta-1)\rho^{\delta-1}u_r+(\delta-1)\frac{2}{r}\rho^{\delta-1}u, ]$$ we have
\begin{equation}\label{L6}
	\begin{aligned}
		L_6=&2\delta\int r^2(v-u)_t u_ru_t=C\int r^2(\rho^{\delta-1})_{rt} u_ru_t\\
		=&-C\int 2r(\rho^{\delta-1})_{t} u_ru_t-C\int r^2(\rho^{\delta-1})_{t} u_{rr}u_t-C\int r^2(\rho^{\delta-1})_{t} u_ru_{rt}\\
		\leq&C|L_5|+C\int |r^2(\rho^{\delta-1})_t \frac{2}{r^2}uu_t|+C\int |r^2(\rho^{\delta-1})_{t} u_ru_{rt}|\\
		\leq& C|L_5|+C\int |[(\rho^{\delta-1})_r u+ (\delta-1)\rho^{\delta-1}u_r+(\delta-1)\frac{2}{r}\rho^{\delta-1}u ]  u u_t|\\
		&+\varepsilon \int r^2 \rho^{\delta-1}u_{rt}^2+C\int r^2\rho^{1-\delta} |[(\rho^{\delta-1})_r u+ (\delta-1)\rho^{\delta-1}u_r+(\delta-1)\frac{2}{r}\rho^{\delta-1}u ]^2u_r^2\\
		\leq&C\Vert \rho^{\delta-2}\rho_r \Vert_{ L^\infty}\Vert u \Vert_{ L^\infty}\Vert ru \Vert_{ L^2}\Vert r u_t \Vert_{ L^2}+\varepsilon \int \rho^{\delta-1}u_t^2+C\Vert r\rho^{\frac{\delta-1}{2}}u_r\Vert_{ L^2}^2\Vert u\Vert_{ L^\infty}^2+C\Vert \rho^{\frac{\delta-1}{2}}u\Vert_{ L^2}^2\Vert u\Vert_{ L^\infty}^2\\
		&+\varepsilon \int r^2 \rho^{\delta-1}u_{rt}^2+C\Vert \rho^{1-\delta}\Vert_{ L^\infty}\Vert \rho^{\delta-2}\rho_r \Vert_{ L^\infty}^2\Vert u \Vert_{ L^\infty}^2\Vert ru_r \Vert_{ L^2}^2+C\Vert r\rho^{\frac{\delta-1}{2}}u_r\Vert_{ L^2}^2\Vert u_r\Vert_{ L^\infty}^2\\
		&+C\Vert r\rho^{\frac{\delta-1}{2}}u_r\Vert_{ L^2}^2\Vert u\Vert_{ L^\infty}^2+C|L_5|\\
		\leq& C(1+C\Vert r \rho^{\frac{\delta-1}{2}}u_r\Vert_{ L^2}^2)\Vert ru_t \Vert_{ L^2}^2+C\Vert r \rho^{\frac{\delta-1}{2}}u_r\Vert_{ L^2}^2+C\Vert \rho^{\frac{\delta-1}{2}}u\Vert_{ L^2}^2+\varepsilon \int r^2 \rho^{\delta-1}u_{rt}^2+\varepsilon \int \rho^{\delta-1}u_t^2\\&+C(1+|L_5|),
	\end{aligned}
\end{equation} where we used the fact that
\begin{equation}\label{urLinfty}
	\Vert u_r \Vert_{ L^\infty}^2\leq \Vert u_r \Vert_{ L^2}^2+\Vert u_{rr}\Vert_{ L^2}^2\leq C+C\Vert ru_t \Vert_{ L^2}^2.
\end{equation}

For $L_7$, with Lemma \ref{le:upper bound of rho} and Corollary \ref{coro:hlinfty}, we have
\begin{equation}\label{L7}
	\begin{aligned}
		L_7=&C\int r^2 (v-u) u_{rt}u_t\leq \varepsilon \int r^2 \rho^{\delta-1}u_{rt}^2+C\int r^2 \rho^{1-\delta} (v-u)^2 u_t^2\\
		\leq& \varepsilon \int r^2 \rho^{\delta-1}u_{rt}^2+C\Vert ru_t \Vert_{ L^2}^2.
	\end{aligned}
\end{equation}

For $L_8$, with $L_{5}$, $L_6$ and Lemma \ref{le:upper bound of rho}, \ref{le:uLP} and \ref{le:urL2}, we have
\begin{equation}\label{L8}
	\begin{aligned}
		L_8=&2(\delta-1)\int r^2(v_t-u_t)\frac{2}{r}uu_t=C\int r(\rho^{\delta-1})_{rt} uu_t\\
		=&-C\int (\rho^{\delta-1})_{t} uu_t-C\int r(\rho^{\delta-1})_{t} u_ru_t-C\int r(\rho^{\delta-1})_{t} uu_{rt}\\
		\leq&C(1+\Vert r \rho^{\frac{1-\delta}{2}}u_r\Vert_{ L^\infty}^2)\Vert ru_{t}\Vert_{ L^2}^2+\varepsilon\int \rho^{\delta-1}u_t^2+C\\
		&+\varepsilon\int r^2\rho^{\delta-1}u_{rt}^2+C\Vert \rho^{1-\delta}[(\rho^{\delta-1})_r u+ (\delta-1)\rho^{\delta-1}u_r+(\delta-1)\frac{2}{r}\rho^{\delta-1}u ]^2u^2\Vert_{ L^1}\\
		\leq&C(1+\Vert r \rho^{\frac{1-\delta}{2}}u_r\Vert_{ L^\infty}^2)\Vert ru_{t}\Vert_{ L^2}^2+\varepsilon\int \rho^{\delta-1}u_t^2+C+\varepsilon\int r^2\rho^{\delta-1}u_{rt}^2\\
		&+C\Vert \rho^{1-\delta}\Vert_{ L^\infty}\Vert \rho^{\delta-2}\rho_r \Vert_{ L^\infty}^2\Vert u \Vert_{ L^\infty}^2\Vert u_r \Vert_{ L^2}^2+C\Vert r\rho^{\frac{\delta-1}{2}}u_r\Vert_{ L^2}^2\Vert u\Vert_{ L^\infty}^2+C\Vert \rho^{\frac{\delta-1}{2}}u\Vert_{ L^2}^2\Vert u\Vert_{ L^\infty}^2\\
		\leq&C(1+\Vert r \rho^{\frac{1-\delta}{2}}u_r\Vert_{ L^\infty}^2)\Vert ru_{t}\Vert_{ L^2}^2+C\Vert r\rho^{\frac{\delta-1}{2}}u_r\Vert_{ L^2}^2+C\Vert \rho^{\frac{\delta-1}{2}}u\Vert_{ L^2}^2+\varepsilon\int r^2\rho^{\delta-1}u_{rt}^2+\varepsilon\int \rho^{\delta-1}u_t^2+C.
	\end{aligned}
\end{equation}

For $L_9$, with Lemma \ref{le:upper bound of rho} and Corollary \ref{coro:hlinfty},we have
\begin{equation}\label{L9}
	\begin{aligned}
		L_9=&2(\delta-1)\int r^2(v-u)\frac{2}{r}u_t^2\leq C\Vert ru_t\Vert_{ L^2}^2.
	\end{aligned}
\end{equation}

Finally, together with \eqref{good term}, \eqref{utL2-1}, $L_1-L_9$, Lemma \ref{le:uLP}, Gronwall's inequality and the compatibility conditions, we have
$$
\Vert r u_t \Vert_{ L^2}^2+\int_0^t\int r^2\rho^{\delta-1}u_{rt}^2drds \leq C,
$$
and then with \eqref{good term}$_2$, we obtain that
\begin{equation}\label{urupper}	
	\Vert ru_{rr}\Vert_{ L^2}+\Vert u_r \Vert_{ L^\infty}\leq C,
\end{equation} and then, by virtue of \eqref{vt} and \eqref{urupper}, we can easily obtain that
\begin{equation}\label{vr L2}
	\Vert rv_r \Vert_{ L^2} \leq C.
\end{equation}

We also have the following fact that
$$
\int_0^t\int r^2 u_{rt}^2drds=\int_0^t\int r^2 \rho^{1-\delta}\rho^{\delta-1}u_{rt}^2drds\leq C\int_0^t\int r^2 \rho^{\delta-1}u_{rt}^2drds\leq C,
$$
and,
\begin{equation}\label{urrrL2}
	\begin{aligned}
		&\int_0^t \Vert ru_{rrr}\Vert_{ L^2}^2(s)ds\\
		\sim& \int_0^t \Vert r \left(\rho^{1-\delta}u_t+\rho^{1-\delta} uu_r+\rho^{\gamma+1-2\delta}(v-u) - (\frac{2}{r}u)_r-\rho^{1-\delta}vu_r+\rho^{1-\delta}(v-u)\frac{2}{r}u\right)_r \Vert_{ L^2}^2(s)ds\\
		\sim&C+ \int_0^t \Vert r u_{rt}\Vert_{ L^2}^2(s)ds\leq C.
	\end{aligned}
\end{equation}

Then, we have proved this lemma.
\end{proof}

\begin{lemma}\label{le:rho gamma-1}
For $\gamma\ge1$, $\delta$ satisfying \eqref{delta} and any $t\in [0,T]$, there holds
\begin{equation}\label{rho gamma-1}
\Vert r\rho(\cdot,t)\Vert_{H^2}+\Vert r\rho_t(\cdot,t) \Vert_{H^1}\leq C.
\end{equation}
\end{lemma}
\begin{proof}
	We can use continuity equation \eqref{system1D}$_1$, Lemma \ref{le:uLP}, \ref{le:urL2} and \ref{le:utL2} obtain this lemma, we omit the details here.
\end{proof}

\subsection*{Proof of Theorem \ref{main result}}

With the preparations above, we complete the 
proof of Theorem \ref{main result}. By Theorem \ref{localthm}, there exists a maximal time $T^*>0$ such that 
\eqref{system1D} and \eqref{initial} admits a unique strong solution $(\rho,u)$ 
on $I_a\times[0,T^*)$ with the regularity stated in \eqref{regularity-2}. 
Then, we let $T^*<+\infty$. However, the global 
{ \it{a priori}} estimates obtained in Lemma \ref{rho L1 estimate}-\ref{le:rho gamma-1}, 
together with Remark \ref{re1}, are 
independent of $T^*$. Hence the solution 
$(\rho,u)$ can be extended beyond $T^*$ by applying Theorem 
\ref{localthm} again, contradicting the maximality of $T^*$. 
Therefore $T^*=+\infty$, and the solution is global-in-time. The 
uniqueness is standard and follows from the local uniqueness. 
This completes the proof of Theorem \ref{main result}. Please refer to Section 4.3 in \cite{W-Z SIAM 2025} for more details.

\section*{Acknowledgement}
I would like to express my sincere gratitude to Professor Huanyao Wen and Professor Xianpeng Hu for their valuable guidance and fruitful academic exchanges throughout this work.

\end{document}